\documentclass[11pt]{amsart}
\usepackage{amsmath, amssymb}
\usepackage{amsfonts}
\usepackage{mathrsfs}
\usepackage[arrow,matrix,curve,cmtip,ps]{xy}
\usepackage{xcolor}
\usepackage{amsthm}
\usepackage{centernot}

\allowdisplaybreaks

\newtheorem{theorem}{Theorem}[section]
\newtheorem{lemma}[theorem]{Lemma}
\newtheorem{proposition}[theorem]{Proposition}
\newtheorem{corollary}[theorem]{Corollary}

\newtheorem*{theorem*}{Theorem}
\theoremstyle{remark}
\newtheorem{remark}[theorem]{Remark}
\newtheorem{definition}[theorem]{Definition}
\newtheorem{example}[theorem]{Example}

\numberwithin{equation}{section}

\newcommand{\Z}{\mathbb{Z}}
\newcommand{\N}{\mathbb{N}}
\newcommand{\R}{\mathbb{R}}
\newcommand{\C}{\mathbb{C}}
\newcommand{\Q}{\mathbb{Q}}

\newcommand{\im}{\operatorname{im }}
\newcommand{\coker}{\operatorname{coker }}

\newcommand{\aut}{\operatorname{Aut}}
\newcommand{\rank}{\operatorname{rank}}
\newcommand{\reg}{\textnormal{reg}}

\newcommand{\sgn}{\operatorname{sgn}}

\newcommand{\gae}{\lower 2pt \hbox{$\, \buildrel {\scriptstyle >}\over {\scriptstyle
\sim}\,$}}

\newcommand{\lae}{\lower 2pt \hbox{$\, \buildrel {\scriptstyle <}\over {\scriptstyle
\sim}\,$}}

\newcommand{\MU}[1]{
\setbox0\hbox{$#1$}
\setbox1\hbox{$W$}
\ifdim\wd0>\wd1 #1^{\sim} \else \widetilde{#1} \fi
}

\begin{document}
\title[Classification of unital simple Leavitt path algebras]{Classification of unital simple Leavitt path algebras of infinite graphs}

	\author{Efren Ruiz}
        \address{Department of Mathematics\\University of Hawaii,
Hilo\\200 W. Kawili St.\\
Hilo, Hawaii\\
96720-4091 USA}
        \email{ruize@hawaii.edu}
        
         \author{Mark Tomforde}
        \address{Department of Mathematics\\University of Houston\\
Houston, Texas\\
77204- 3008, USA}
        \email{tomforde@math.uh.edu}
        
\date{\today}


\thanks{The second author was supported by a grant from the Simons Foundation (\#210035 to Mark Tomforde)}

\date{\today}

\begin{abstract}
We prove that if $E$ and $F$ are graphs with a finite number of vertices and an infinite number of edges, if $K$ is a field, and if $L_K(E)$ and $L_K(F)$ are simple Leavitt path algebras, then $L_K(E)$ is Morita equivalent to  $L_K(F)$ if and only if $K_0^\textnormal{alg} (L_K(E)) \cong K_0^\textnormal{alg} (L_K(F))$ and the graphs $E$ and $F$ have the same number of singular vertices, and moreover, in this case one may transform the graph $E$ into the graph $F$ using basic moves that preserve the Morita equivalence class of the associated Leavitt path algebra.  We also show that when $K$ is a field with no free quotients, the condition that $E$ and $F$ have the same number of singular vertices may be replaced by $K_1^\textnormal{alg} (L_K(E)) \cong K_1^\textnormal{alg} (L_K(F))$, and we produce examples showing this cannot be done in general.  We describe how we can combine our results with a classification result of Abrams, Louly, Pardo, and Smith to get a nearly complete classification of unital simple Leavitt path algebras --- the only missing part is determining whether the ``sign of the determinant condition" is necessary in the finite graph case.  We also consider the Cuntz splice move on a graph and its effect on the associated Leavitt path algebra.  \end{abstract}

\maketitle

%%%%%%%%%%%%%%%%%%%%%%%%%%%%%%%%%%%%%%%%%%%%%%%%%%%%%%
\section{Introduction}
%%%%%%%%%%%%%%%%%%%%%%%%%%%%%%%%%%%%%%%%%%%%%%%%%%%%%%

A current program in the theory of Leavitt path algebras is to classify the unital simple Leavitt path algebras up to Morita equivalence in terms of an invariant that can be easily calculated from the graph.  Much of the inspiration for this program comes from the theory of $C^*$-algebras, where it has been found that various $C^*$-algebras are classified up to Morita equivalence by $K$-theory.  Since Leavitt path algebras are a class of algebras that are very similar to $C^*$-algebras (particularly the graph $C^*$-algebras), it is natural to ask whether the Leavitt path algebras may be classified by $K$-theoretic invariants and whether this classification may be obtained by techniques similar to those used to classify $C^*$-algebras.  While topological $K$-theory is often a useful invariant for $C^*$-algebras, since it can be computed for many $C^*$-algebras of interest, algebraic $K$-theory is notoriously difficult to compute --- indeed, it is still not known what all the algebraic $K$-groups of the ring $\Z$ are.  Nonetheless, algebraic $K$-theory for Leavitt path algebras is much more tractable.  Indeed, due to Ara, Brustenga, and Corti\~nas \cite{ABC}, there are explicit formulae for determining $K_0^\textnormal{alg} (L_K(E))$ and $K_1^\textnormal{alg} (L_K(E))$ for any graph $E$ and any field $K$, and there is a long exact sequence relating the higher algebraic $K$-groups of $L_K(E)$ and the algebraic $K$-groups of the field $K$.  Consequently, it is often the case that much of the algebraic $K$-theory of $L_K(E)$ can be computed from the graph $E$ and the field $K$.

Any unital simple Leavitt path algebra is either purely infinite or isomorphic to a matrix algebra, so the nontrivial part of the classification program mentioned above involves the classification of unital, purely infinite, simple Leavitt path algebras.  In the theory of $C^*$-algebras, the main result for the classification of purely infinite, simple $C^*$-algebras is the Kirchberg-Phillips classification theorem, which states that for a wide class of purely infinite, simple $C^*$-algebras the pair consisting of the $K_0^\textnormal{top}$-group and the $K_1^\textnormal{top}$-group is a complete Morita equivalence invariant.  In particular, the Kirchberg-Phillips classification theorem applies to all graph $C^*$-algebras that are simple and purely infinite, and this begs the question of whether a similar result is true for Leavitt path algebras.  Unfortunately, there is little hope for the techniques used to prove the Kirchberg-Phillips classification theorem to work for Leavitt path algebras, because the technology involved relies heavily on the $C^*$-algebra structure.  On the other hand, there are many preliminary classification results, which are now subsumed by the Kirchberg-Phillips classification theorem, that were originally proven by vastly different techniques and seem more applicable to Leavitt path algebras.  One such example is the classification of the simple Cuntz-Krieger algebras.  The Cuntz-Krieger algebras coincide with the $C^*$-algebras of finite graphs with no sinks or sources, and hence are $C^*$-algebraic analogues of Leavitt path algebras of finite graphs with no sinks or sources.  The simple Cuntz-Krieger algebras were classified by work of R\o rdam \cite{Ro}, Cuntz \cite[Appendix~7]{Ro}, and Cuntz and Krieger \cite{CK}, and it was shown that the the $K^\textnormal{top}_0$-group is a complete Morita equivalence invariant.  (It turns out that for Cuntz-Krieger algebras, the $K^\textnormal{top}_1$-group is redundant and can be recovered from the $K_0^\textnormal{top}$-group.)  This classification relied on connections with symbolic dynamics, and the observation that Morita equivalence of two simple Cuntz-Krieger algebras is related to flow equivalence of the associated shift spaces of the graphs.  It was shown that graphs with flow equivalent shift spaces produce Morita equivalent $C^*$-algebras, and in addition (due to the key work of R\o rdam and Cuntz in \cite{Ro}) it is possible for certain graphs with shift spaces that are not flow equivalent to have Morita equivalent $C^*$-algebras.  Moreover, the proofs showed that for Cuntz-Krieger algebras, the Morita equivalence can be realized concretely in the sense that there is a set of basic moves that will turn one graph into another while preserving the Morita equivalence of the associated $C^*$-algebra.  These moves include familiar moves from symbolic dynamics that preserve flow equivalence of the shift space, as well as an additional move, called the ``Cuntz splice", that does not preserve flow equivalence of the shift space.

Abrams, Louly, Pardo, and Smith showed that for Leavitt path algebras of finite graphs with no sinks, similar connections with symbolic dynamics exist and a classification up to Morita equivalence can be obtained \cite[Theorem~1.25]{ALPS}.  However, unlike in the $C^*$-algebra case, they were unable to determine if the ``Cuntz splice" move preserves Morita equivalence of the associated Leavitt path algebra.  To obtain their result, they therefore needed to impose the additional condition that $\sgn (\det (I-A^t_E)) = \sgn (\det (I-A^t_F))$, where $A_E$ and $A_F$ are the vertex matrices of $E$ and $F$, respectively.  Thus their result states that if $E$ and $F$ are finite graphs with no sinks, and if $L_K(E)$ and $L_K(F)$ are simple algebras, then $K_0^\textnormal{alg} (L_K(E)) \cong K^\textnormal{alg}_0(L_K(E) )$ and $\sgn (\det (I-A^t_E)) = \sgn (\det (I-A^t_F))$ implies $L_K(E)$ is Morita equivalent to $L_K(F)$,  and moreover, in this case $E$ may be transformed into $F$ by a sequence of basic moves.  It is currently unknown --- and the focus of much investigation --- whether the ``sign of the determinant condition" is necessary.  It is also interesting to note that the $K^\textnormal{alg}_0$-group is the only $K$-theoretic data needed, since \emph{a priori} it is not clear in this situation that the higher $K$-groups $K_n^\textnormal{alg}(L_K(E))$ for $n \geq 1$ are determined by the group $K_0^\textnormal{alg}(L_K(E))$.

Recently, S\o rensen turned his attention toward producing additional classifications of graph $C^*$-algebras using techniques from dynamics and moves on the graphs.  S\o rensen showed in \cite{Sor} that if $E$ and $F$ are graphs with a finite number of vertices and an infinite number of edges, and if $C^*(E)$ and $C^*(F)$ are simple $C^*$-algebras, then the graph $C^*$-algebras $C^*(E)$ and $C^*(F)$ are Morita equivalent if and only if $K_0^\textnormal{top} (C^*(E)) \cong K_0^\textnormal{top} (C^*(F))$ and $K_1^\textnormal{top} (C^*(E)) \cong K_1^\textnormal{top} (C^*(F))$, and moreover, in this case one may transform the graph $E$ into the graph $F$ using basic moves that do not require the Cuntz splice.  (We mention that the $K_1^\textnormal{top}$-group is a necessary part of the invariant in this case and it cannot be recovered from the $K_0^\textnormal{top}$-group.)

The inspiration for this paper comes from S\o rensen's result --- particularly the fact that the Cuntz splice is not needed in his classification.  Building off S\o rensen's result, we are able to show in Theorem~\ref{field-does-not-matter-thm} that if $E$ and $F$ are graphs with a finite number of vertices and an infinite number of edges, and if $L_K(E)$ and $L_K(F)$ are simple Leavitt path algebras, then $L_K(E)$ and $L_K(F)$ are Morita equivalent if and only if $K_0^\textnormal{alg} (L_K(E)) \cong K_0^\textnormal{alg} (L_K(F))$ and the graphs $E$ and $F$ have the same number of singular vertices; and moreover, in this case $E$ can be transformed into $F$ using the basic moves coming from the study of flow equivalence of shift spaces.  We also show that when $K$ is a field with no free quotients (see Definition~\ref{nfq-group-def} and Definition~\ref{nfq-field-def}), then the condition in this result that the graphs $E$ and $F$ have the same number of singular vertices may replaced by the condition that $K_1^\textnormal{alg} (L_K(E)) \cong K_1^\textnormal{alg} (L_K(F))$, and we produce examples in Section~\ref{example-sec} showing that this replacement cannot be done for Leavitt path algebras over general fields.

Our classification result has a few remarkable properties: First, it is interesting that the underlying field $K$ plays a role in the classification, since the field often does not affect Leavitt path algebras results and, in particular, seems irrelevant in the classification of Leavitt path algebras of finite graphs described in  \cite[Theorem~1.25]{ALPS}.  Second, it is convenient that the complete Morita equivalence invariant consists of the $K_0^\textnormal{alg}$-group and the number of singular vertices (or, in the case the underlying field has no free quotients, the $K_0^\textnormal{alg}$-group and $K_1$-group), since for Leavitt path algebras we have explicit formulae for the $K_0^\textnormal{alg}$-group and the $K_1^\textnormal{alg}$-group, but in general we cannot explicitly compute the higher $K_n^\textnormal{alg}$-groups for $n \geq 2$ (see Proposition~\ref{K-theory-comp-prop}).  Third, it is fortunate that the Cuntz splice is not needed among the basic moves required to transform one graph into another (just as in S\o rensen's result for graph $C^*$-algebras), so that we are able to circumvent the need for any kind of ``sign of the determinant condition" and thereby obtain necessary and sufficient conditions for Morita equivalence in this setting.

While our main purpose in this paper is to establish the classification of simple Leavitt path algebras of graphs with a finite number of vertices and infinite number of edges, we also pursue several other lines of investigation.  One of our main questions is: To what extent do the algebraic $K$-groups of a Leavitt path algebra determine the that Leavitt path algebra's structure?  In particular, we consider what information is contained in the $K_0^\textnormal{alg}$-group and $K_1^\textnormal{alg}$-group of a Leavitt path algebra, and ask when these groups will determine the higher $K_n^\textnormal{alg}$-groups.  We also explore the question of when isomorphic $K^\textnormal{alg}$-groups of two Leavitt path algebra will imply isomorphic $K^\textnormal{top}$-groups of the corresponding graph $C^*$-algebras.  We will see that the condition that the underlying field has no free quotients is important in all of these questions.

This paper is organized as follows.  We begin with some preliminaries in Section~\ref{prelim-sec}.  In Section~\ref{moves-sec} we introduce our basic moves, referred to as Moves~(S), (O), (I), and (R), and prove that they preserve Morita equivalence of the associated Leavitt path algebra.  While special cases of these moves have been considered by other authors (e.g., when performed on finite graphs), here we describe them in full generality, and discuss what is allowed when there are infinite emitters in the graph.  In Section~\ref{simple-graph-sec} we discuss the graphs associated with simple Leavitt path algebras and simple graph $C^*$-algebras, and we refer to these as \emph{simple graphs}.  In Section~\ref{compute-K-sec} we describe how to compute the topological $K$-theory of a graph $C^*$-algebra and the algebraic $K$-theory of a Leavitt path algebra.  In particular, we describe explicit formulae for the $K^\textnormal{top}_0$-group and $K^\textnormal{top}_1$-group of a graph $C^*$-algebra, describe explicit formulae for the $K^\textnormal{alg}_0$-group and $K^\textnormal{alg}_1$-group of a Leavitt path algebra, and describe how the higher algebraic $K$-groups of a Leavitt path algebra fit into a long exact sequence that involves the algebraic $K$-groups of the underlying field.  Additionally, in a number of corollaries we examine the computation of the $K$-groups in the special case that the graph has finitely many vertices.  In Section~\ref{K-theory-comparison-sec} we introduce a property for groups and fields that we call \emph{having no free quotients}.  We examine equivalent ways to formulate this property, and also consider a number of examples.  We prove in Theorem~\ref{LPA-K-implies-C-star-K-thm} that if $E$ and $F$ are graphs and $K$ is a field with no free quotients, then $K_n^\textnormal{alg}(L_K(E)) \cong K_n^\textnormal{alg}(L_K(F))$ for $n = 0,1$ implies $K_n^\textnormal{top}(C^*(E)) \cong K_n^\textnormal{top}(C^*(F))$ for $n = 0,1$.  We also examine the information contained in the algebraic $K$-groups, and show that for a graph with finitely many vertices and a field with no free quotients, knowing the $K_0^\text{alg}$-group and the $K_1^\text{alg}$-group of the Leavitt path algebra is equivalent to knowing the $K_0^\text{alg}$-group of the Leavitt path algebra and the the number of singular vertices in the graph --- but that this equivalence does not hold if the hypothesis that the field has no free quotients is dropped.  In Section~\ref{class-thm-sec} we consider the classification of simple Leavitt path algebras $L_K(E)$ when $E$ is a graph with a finite number of vertices and an infinite number of edges.  We prove in Theorem~\ref{class-unital-inf-LPA-thm} that when $K$ is a field with no free quotients, then $(K_0^\textnormal{alg} (L_K(E)), K_1^\textnormal{alg} (L_K(E)))$ is a complete Morita equivalence invariant for this class.  In Theorem~\ref{field-does-not-matter-thm} we show if $K$ is any field, then $(K_0^\textnormal{alg} (L_K(E)), |E^0_\textnormal{sing}|)$ is a complete Morita equivalence invariant for this class (where $|E^0_\textnormal{sing}|$ denotes the number of singular vertices in $E$), and that when two such Leavitt path algebras $L_K(E)$ and $L_K(F)$ are Morita equivalent, the graph $E$ may be transformed into the graph $F$ using the basic moves (S), (O), (I), (R), and their inverses.  We also show that if $E$ and $F$ are simple graphs with a finite number of vertices and an infinite number of edges and $K$ is any field, then $C^*(E)$ is strongly Morita equivalent to $C^*(F)$ if and only if $L_K(E)$ is Morita equivalent to $L_K(F)$.

In Section~\ref{unital-class-sec} we combine the results of Theorem~\ref{class-unital-inf-LPA-thm} and Theorem~\ref{field-does-not-matter-thm} with the classification result of Abrams, Louly, Pardo, and Smith in \cite[Theorem~1.25]{ALPS} to give the current status of the program to classify unital simple Leavitt path algebras up to Morita equivalence.  We show that the only remaining piece to complete this program is to determine if the ``sign of the determinant condition" is necessary in the finite graph case.  In Section~\ref{CS-sec} we describe the \emph{Cuntz splice} move for Leavitt path algebras.  We describe how determining if the ``sign of the determinant condition" is necessary is equivalent to determining if the Cuntz splice move preserves the Morita equivalence of the associated Leavitt path algebra.  We also show that for any graph $E$ and any field $K$, the Cuntz splice preserves the isomorphism class of the $K_0^\textnormal{alg}$-group and the $K_1^\textnormal{alg}$-group of the associated Leavitt path algebra, and in the finite graph case the Cuntz splice flips the sign of $\det (I-A_E^t)$.  In addition, we discuss the fact that when the Cuntz splice is applied to a simple graph with finitely many vertices and infinitely many edges, it may instead be obtained by a finite sequence of the moves (S), (O), (I), (R), and their inverses.  In Example~\ref{E-infty-CS-ex} we take the graph $E_\infty$ with one vertex and countably many edges, perform the Cuntz splice to obtain a graph $E_\infty^-$, and then show how we may transform $E_\infty$ into $E_\infty^-$ using the moves (O), (I), (R), and their inverses.  In Section~\ref{isomorphism-sec} we consider classification up to isomorphism, and show that if $E$ is a simple graph with a finite number of vertices and an infinite number of edges, $K$ is any field, and if $[1_{L_K(E)}]_0$ is an automorphism invariant element of $K_0^\textnormal{alg}(L_K(E))$, then $((K_0^\textnormal{alg}(L_K(E)), [1_{L_K(E)}]_0), |E^0_\textnormal{sing}|)$ is a complete isomorphism invariant.  In Section~\ref{example-sec} we consider some interesting (counter)examples.  In particular, we observe that the field $\Q$ does not satisfy the hypothesis of having no free quotients, and we give an example of simple graphs $E$ and $F$ with a finite number of vertices and infinite number of edges such $K_0^\text{alg} (L_\Q(E)) \cong K_0^\text{alg} (L_\Q(F))$ and $K_1^\text{alg} (L_\Q(E)) \cong K_1^\text{alg} (L_\Q(F))$, but $L_\Q(E)$ and $L_\Q(F)$ are not Morita equivalent.  We also see that our graphs have different numbers of singular vertices despite the fact their $K_0^\text{alg}$-groups and $K_1^\text{alg}$-groups are isomorphic.  We also produce a similar example showing that over the field $\Q$ it is possible for a finite graph and a graph with a finite number of vertices and infinite number of edges to have simple Leavitt path algebras that are not Morita equivalent but have isomorphic $K_0^\text{alg}$-groups and $K_1^\text{alg}$-groups.  These examples show that the invariant $(K_0^\textnormal{alg} (L_K(E)), |E^0_\textnormal{sing}|)$ must be used instead of $(K_0^\textnormal{alg} (L_K(E)), K_1^\textnormal{alg} (L_K(E)))$ when the underlying field has free quotients.  We conclude the paper by examining our examples in greater detail, discussing their consequences, and raising a number of questions for further study.

\smallskip

\noindent \textsc{Terminology:} Throughout this paper the term \emph{countable} shall mean finite or countably infinite.  We will make the standing assumption that all of our graphs are countable.  If $G$  is a finitely generated abelian group, the \emph{rank} of $G$ shall mean the torsion-free rank; i.e., if $G = \Z_{n_1}\oplus \cdots \oplus \Z_{n_k} \oplus \Z^m$, then $\rank G = m$.  We refer to Morita equivalence in the category of $C^*$-algebras as \emph{strong Morita equivalence} to distinguish it from Morita equivalence of rings.  We will also make use of both algebraic and topological $K$-theory throughout this paper.  To distinguish the two, we use the notation $K_n^\textnormal{top}$ for the topological $K$-groups of a $C^*$-algebra and the notation $K_n^\textnormal{alg}$ for the algebraic $K$-groups of a ring.  

\smallskip

\noindent \textsc{Acknowledgements:} We thank Adam S\o rensen for many useful discussions related to this work.  In particular, we thank him for describing the sequence of moves we reproduce in Example~\ref{E-infty-CS-ex}, and for helping us to simplify the graphs we use in Example~\ref{K-theory-counter-ex}.   We thank Roozbeh Hazrat for pointing out some miscalculations in a previous version of this paper.  We also thank Gene Abrams for useful discussions and for helping us to strengthen the result of Proposition~\ref{unital-classification-zero-class-prop}.  We also extend a profound thank you to Enrique Pardo, who after looking at a previous version of this paper pointed out Lemma~\ref{l:fieldextension}, which allowed us to strengthen our previous results and also deduce that $(K_0^\textnormal{alg}(L_K(E)), |E^0_\textnormal{sing}|)$ is a complete Morita equivalence invariant for Leavitt path algebras over general fields.

%%%%%%%%%%%%%%%%%%%%%%%%%%%%%%%%%%%%%%%%%%%%%%%%%%%%%%
\section{Preliminaries} \label{prelim-sec}
%%%%%%%%%%%%%%%%%%%%%%%%%%%%%%%%%%%%%%%%%%%%%%%%%%%%%%

In this section we establish notation and recall some standard definitions.  

\begin{definition}
A \emph{graph} $(E^0, E^1, r, s)$ consists of a countable set $E^0$ of vertices, a countable set $E^1$ of edges, and maps $r : E^1 \to E^0$ and $s : E^1 \to E^0$ identifying the range and source of each edge.  \end{definition}

\begin{definition}
Let $E := (E^0, E^1, r, s)$ be a graph. We say that a vertex $v
\in E^0$ is a \emph{sink} if $s^{-1}(v) = \emptyset$, and we say
that a vertex $v \in E^0$ is an \emph{infinite emitter} if
$|s^{-1}(v)| = \infty$.  A \emph{singular vertex} is a vertex that
is either a sink or an infinite emitter, and we denote the set of
singular vertices by $E^0_\textnormal{sing}$.  We also let
$E^0_\textnormal{reg} := E^0 \setminus E^0_\textnormal{sing}$, and
refer to the elements of $E^0_\textnormal{reg}$ as \emph{regular
vertices}; i.e., a vertex $v \in E^0$ is a regular vertex if and
only if $0 < |s^{-1}(v)| < \infty$.  A graph is \emph{row-finite}
if it has no infinite emitters.  A graph is \emph{finite} if both
sets $E^0$ and $E^1$ are finite (or equivalently, when $E^0$ is
finite and $E$ is row-finite).
\end{definition}

\begin{definition}
If $E$ is a graph, a \emph{path} is a sequence $\alpha := e_1 e_2
\ldots e_n$ of edges with $r(e_i) = s(e_{i+1})$ for $1 \leq i \leq
n-1$.  We say the path $\alpha$ has \emph{length} $| \alpha| :=n$,
and we let $E^n$ denote the set of paths of length $n$.  We
consider the vertices in $E^0$ to be paths of length zero.  We
also let $E^* := \bigcup_{n=0}^\infty E^n$ denote the paths of
finite length in $E$, and we extend the maps $r$ and $s$ to $E^*$
as follows: For $\alpha = e_1 e_2 \ldots e_n \in E^n$ with $n\geq
1$, we set $r(\alpha) = r(e_n)$ and $s(\alpha) = s(e_1)$; for
$\alpha = v \in E^0$, we set $r(v) = v = s(v)$.  A \emph{cycle} is a path $\alpha=e_1 e_2 \ldots e_n$ with length $|\alpha| \geq 1$ and $r(\alpha) = s(\alpha)$.  
\end{definition}

\begin{definition} \label{graph-C*-def}
If $E$ is a graph, the \emph{graph $C^*$-algebra} $C^*(E)$ is the universal
$C^*$-algebra generated by mutually orthogonal projections $\{ p_v
: v \in E^0 \}$ and partial isometries with mutually orthogonal
ranges $\{ s_e : e \in E^1 \}$ satisfying
\begin{enumerate}
\item $s_e^* s_e = p_{r(e)}$ \quad  for all $e \in E^1$
\item $s_es_e^* \leq p_{s(e)}$ \quad for all $e \in E^1$
\item $p_v = \sum_{\{ e \in E^1 : s(e) = v \}} s_es_e^* $ \quad for all $v \in E^0_\textnormal{reg}$.
\end{enumerate}
\end{definition}
For a path $\alpha = e_1 \ldots e_n$ of positive length we let $s_\alpha := s_{e_1} \ldots s_{e_n}$, and for a vertex $v$ we define $s_v := p_v$.  One can show that $$C^*(E) = \overline{\operatorname{span}} \{ s_\alpha s_\beta^* : \text{$\alpha$ and $\beta$ are paths in $E$ with $r(\alpha) = r(\beta)$} \}.$$
If $E$ is a graph, the graph $C^*$-algebra $C^*(E)$ is unital if and only if the vertex set $E^0$ is finite, in which case $1 = \sum_{v \in E^0} p_v$.

\begin{definition} \label{LPA-lin-invo-def}
Let $E$ be a graph, and let $K$ be a field. We let $(E^1)^*$
denote the set of formal symbols $\{ e^* : e \in E^1 \}$.  The \emph{Leavitt path
algebra of $E$ with coefficients in $K$}, denoted $L_K(E)$,  is
the free associative $K$-algebra generated by a set $\{v : v \in
E^0 \}$ of pairwise orthogonal idempotents, together with a set
$\{e, e^* : e \in E^1\}$ of elements, modulo the ideal generated
by the following relations:
\begin{enumerate}
\item $s(e)e = er(e) =e$ for all $e \in E^1$
\item $r(e)e^* = e^* s(e) = e^*$ for all $e \in E^1$
\item $e^*f = \delta_{e,f} \, r(e)$ for all $e, f \in E^1$
\item $v = \displaystyle \sum_{\{e \in E^1 : s(e) = v \}} ee^*$ whenever $v \in E^0_\reg$.
\end{enumerate}
\end{definition}

If $\alpha = e_1 \ldots e_n$ is a path of positive length, we define $\alpha^* = e_n^* \ldots e_1^*$.  One can show that $$L_K(E) = \operatorname{span}_K \{ \alpha \beta^* : \text{$\alpha$ and $\beta$ are paths in $E$ with $r(\alpha) = r(\beta)$} \}.$$  In addition, $L_K(E)$ has a natural $\Z$-grading, with the elements of degree $n$ defined by $$L_K(E)_n =  \operatorname{span}_K \{ \alpha \beta^* : |\alpha| - |\beta| = n \}.$$
If $E$ is a graph and $K$ is a field, the Leavitt path algebra $L_K(E)$ is unital if and only if the vertex set $E^0$ is finite, in which case $1 = \sum_{v \in E^0} v$.

\begin{definition}
If $E$ is a graph, a subset $H \subseteq E^0$ is \emph{hereditary} if whenever $e \in E^1$ and $s(e) \in H$, then $r(e) \in H$.  A hereditary subset $H$ is called \emph{saturated} if $\{ v \in E^0_\reg : r(s^{-1}(v)) \subseteq H \} \subseteq H$. 
\end{definition}

\begin{definition}
Let $E$ be a graph and let $V \subseteq E^{0}$.  We define the \emph{saturation of $V$} to be the set
$$\Sigma (V) := \bigcup_{n=0}^\infty \Sigma_n (V),$$
where $$\Sigma_0 (V) := \{ v \in E^0 : \text{there exists $\alpha \in E^0$ with $s(\alpha) \in V$ and $r(\alpha) = v$} \}$$
and $$\Sigma_{n+1} (V) := \Sigma_n(V) \cup \{ v \in E^0_\textnormal{reg} : r(s^{-1}(v)) \subseteq \Sigma_n(V) \} \qquad \text{for $n=0,1,2,\ldots$} $$  It is straightforward to verify that for any subset $V \subseteq E^0$, the saturation $\Sigma (V)$ is a saturated hereditary set, and moreover, $\Sigma(V)$ is the smallest saturated hereditary subset containing $V$.
\end{definition}

\begin{definition}
Let $E$ be a graph, and let $V \subseteq E^0$.  For any field $K$, we call the subalgebra $VL_K(E)V := \sum_{v,w \in V} v L_K(E) w$ the \emph{corner of $L_K(E)$ determined by $V$}.  We say the corner $VL_K(E)V$ is \emph{full} if $\Sigma (V) = E^0$.
\end{definition}

\begin{lemma}\label{corner-ME-lem}
Let $E$ be a graph, let $V \subseteq E$, and let $K$ be a field. If the corner $VL_K(E)V := \sum_{v,w \in V} v L_K(E) w$ is full, then $VL_K(E)V$ is Morita equivalent to $L_K(E)$.
\end{lemma}

\begin{proof}
If $V$ is finite, let $w := \sum_{v \in V} v$.  Then $L_K(E) w L_K(E) = I_{\Sigma(V)} = I_{E^0} = L_K(E)$.  A computation shows that $$\left( wL_K(E)w , L_K(E) w L_K(E), L_K(E)w, w L_K(E) \right)$$ is a (surjective) Morita context, so that $VL_K(E)V = wL_K(E)w$ is Morita equivalent to $L_K(E) = L_K(E) w L_K(E)$.

If $V$ is infinite, let $\mathcal{F} := \{ F : F \subseteq V \text{ and $F$ is finite} \}$.  For each $F \in \mathcal{F}$ define $w_F := \sum_{v \in F} v$.  Then $\sum_{F \in \mathcal{F}} L_K(E) w_F L_K(E) = I_{\Sigma(V)} = I_{E^0} = L_K(E)$.  A computation shows that 
$$\left( \sum_{F \in \mathcal{F} } w_F L_{K} ( E ) w_F , \sum_{ F \in \mathcal{F} } L_{K} ( E ) w_F L_{K} (E), \sum_{ F \in \mathcal{F} } L_{K} ( E ) w_F , \sum_{ F \in \mathcal{F} } w_F L_{K} ( E ) \right)$$
is a (surjective) Morita context, so that $VL_K(E)V = \sum_{ F \in \mathcal{F}} w_F L_{K} ( E ) w_F$ is Morita equivalent to $L_{K} (E) = \sum_{ F \in \mathcal{F}} L_{K} ( E ) w_F L_{K} (E)$.

\end{proof}

%%%%%%%%%%%%%%%%%%%%%%%%%%%%%%%%%%%%%%%%%%%%%%%%%%%%%%
\section{Moves on Graphs} \label{moves-sec}
%%%%%%%%%%%%%%%%%%%%%%%%%%%%%%%%%%%%%%%%%%%%%%%%%%%%%%

In this section we describe the moves on graphs used in \cite{Sor} and examine the effect of each on the Leavitt path algebra.  In particular, we prove that each move preserves the Morita equivalence class of the associated Leavitt path algebra.  We mention that these moves have been considered by other authors, and were previously noted to preserve the Morita equivalence class of either the associated Leavitt path algebra or the associated graph $C^*$-algebra.  We present self-contained proofs here to give a unified treatment and present the proofs for Leavitt path algebras.  For reference, we point out that Proposition~\ref{Move-S-ME-prop} is similar to \cite[Proposition~1.4]{ALPS}, Proposition~\ref{Move-O-ME-prop} is an analogue of \cite[Theorem~3.2]{BP}, Proposition~\ref{Move-I-ME-prop} is an analogue of \cite[Theorem~5.3]{BP}, and Proposition~\ref{Move-R-ME-prop} appears in \cite[\S3]{Sor}.

\begin{definition}[Move (S): Remove a Regular Source] \label{Move-S-def}
Let $E = (E^0, E^1, r, s)$ be a graph, and let $w \in E^0$ be a source that is also a regular vertex.  The graph $E_S = (E_S^0, E_S^1, r_S, s_S)$ is defined by 
$$E_S^0 := E^0 \setminus \{ w \} \qquad E_S^1 := E^1 \setminus s^{-1}(w) \qquad r_S := r|_{E^0_S} \qquad s_S := s|_{E^0_S}.$$  We call $E_S$ the \emph{graph obtained by removing the source $w$ from $E$}, and say $E_S$ is formed by \emph{performing Move (S)} to $E$.
\end{definition}

\begin{proposition}[Cf.~Proposition~1.4 of \cite{ALPS}] \label{Move-S-ME-prop}
Suppose $E$ is a graph and $K$ is a field.  Let $w \in E^0$ be a source that is also a regular vertex, and set $V := E^0 \setminus \{ w \}$.  Then $V L_K(E) V$ is a full corner of $L_K(E)$ that is isomorphic to $L_K(E_S)$ via a graded isomorphism.  In particular, $L_K(E_S)$ is Morita equivalent to $L_K(E)$.
\end{proposition}

\begin{proof}
Let $\{v, e, e^* : v \in E^0, e \in E^1 \}$ be a generating $E$-family in $L_K(E)$.  For each $v \in E_S^0$ let $P_v := v$, and for each $e \in E_S^1$ let $S_e := e$.  Then $\{P_v, S_e, S_e^* : v \in E_S^0, e \in E_S^1 \}$ is an $E_S$-family in $L_K(E)$, and hence there exists a homomorphism $\phi :  L_K(E_S) \to L_K(E)$ with $\phi (v) = P_v$, $\phi(e) = S_e$, and $\phi(e^*) = S_e^*$ for all $v \in E_S^0$ and all $e \in E_S^1$.  Since $P_v = v$ has degree $0$, $S_e = e$ has degree $1$, and $S_e^* = e^*$ has degree $-1$, we see that $\phi$ is a graded homomorphism.  Therefore, since $\phi(v) = P_v \neq 0$ for all $v \in E_S^0$, the graded uniqueness theorem implies that $\phi$ is injective, and $\phi$ is an isomorphism onto $\im \phi$.  However, if we let $V := E^0 \setminus \{ w \}$, then the corner determined by $V$ is $VL_K(E) V = \im \phi$.  Moreover, since $w$ is a regular vertex and a source, $r(s^{-1}(w)) \subseteq V$, and $\Sigma (V) = E^0$.  Hence $\im \phi$ is a full corner of $L_K(E)$, and by Lemma~\ref{corner-ME-lem} $\im \phi$ is Morita equivalent to $L_K(E)$.
\end{proof}

\begin{definition}[Move (O): Outsplit at a Non-Sink] \label{Move-O-def}
Let $E = (E^0, E^1, r, s)$ be a graph, and let $w \in E^0$ be vertex that is not a sink.   Partition $s^{-1}(w)$ as a disjoint union of a finite number of nonempty sets $$s^{-1}(w) = \mathcal{E}_1 \sqcup \mathcal{E}_2 \sqcup \ldots \sqcup \mathcal{E}_n$$ with the property that at most one of the $\mathcal{E}_i$ is infinite.  The graph $E_O := (E_O^0, E_O^1, r_O, s_O)$ is defined by 
\begin{align*}
E_O^0 &:= \{v^1 : \text{$v \in E^0$ and $v \neq w$} \} \cup \{ w^1, \ldots, w^n \} \\
E_O^1 &:= \{ e^1 : \text{$e \in E^1$ and $r(e) \neq w$} \} \cup \{ e^1, \ldots, e^n : e \in E^1 \text{ and } r(e) = w\} \\
r_{E_O} (e^i) &:= \begin{cases} r(e)^1 & \text{if $e \in E^1$ and $r(e) \neq w$} \\ w^i & \text{if $e \in E^1$ and $r(e) = w$}
 \end{cases} \\
s_{E_O} (e^{i }) &:= \begin{cases} s(e)^1 & \text{if $e \in E^1$ and $s(e) \neq w$} \\ s(e)^j & \text{if $e \in E^1$ and $s(e) = w$ with $e \in \mathcal{E}_j$} 
 \end{cases}
\end{align*}
 We call $E_O$ the \emph{graph obtained by outsplitting $E$ at $w$}, and say $E_O$ is formed by \emph{performing Move (O)} to $E$.
\end{definition}

\begin{proposition}[Cf.~Theorem~3.2 of \cite{BP}] \label{Move-O-ME-prop}
Suppose $E$ is a graph, $w \in E^0$ is vertex that is not a sink, and a partition $$s^{-1}(w) = \mathcal{E}_1 \sqcup \mathcal{E}_2 \sqcup \ldots \sqcup \mathcal{E}_n$$ is chosen with the $\mathcal{E}_i$ disjoint nonempty sets and at most one of the $\mathcal{E}_i$ is infinite.  Then for any field $K$ it is the case that $L_K(E_O)$ is isomorphic to $L_K(E)$ via a graded isomorphism.
\end{proposition}

\begin{proof}
For $v \in E^0$ and $e \in E^1$ define
$$P_v := \begin{cases} v^1 & \text{if $v \neq w$} \\ \sum_{i=1}^n w^i & \text{if $v = w$} \end{cases} \qquad \text{ and } \qquad S_e := \begin{cases} e^1 & \text{if $r(e) \neq w$} \\ \sum_{i=1}^n e^i & \text{if $r(e) = w$.} \end{cases}.$$  It is straightforward to verify that $\{ P_v, S_e, S_e^* : v \in E^0, e \in E^1 \}$ is an $E$-family in $L_K(E_O)$, and hence there exists a homomorphism $\phi :  L_K(E) \to L_K(E_O)$ with $\phi (v) = P_v$, $\phi(e) = S_e$, and $\phi(e^*) = S_e^*$ for all $v \in E^0$ and all $e \in E^1$.  It is also straightforward to verify that $\im \phi = L_K(E_O)$ so that $\phi$ is surjective.  Since $P_v$ has degree $0$, $S_e$ has degree $1$, and $S_e^*$ has degree $-1$, we see that $\phi$ is a graded homomorphism.  Therefore, since $\phi(v) = P_v \neq 0$ for all $v \in E_S^0$, the graded uniqueness theorem implies that $\phi$ is injective, and $\phi$ is a graded isomorphism.
\end{proof}

\begin{definition}[Move (I): Insplit at a Regular Non-Source] \label{Move-I-def}
Suppose that $E = (E^0, E^1, r, s)$ is a graph, and let $w \in E^0$ be a regular vertex that is not a source.   Partition $r^{-1}(w)$ as a disjoint union of a finite number of nonempty sets $$r^{-1}(w) = \mathcal{E}_1 \sqcup \mathcal{E}_2 \sqcup \ldots \sqcup \mathcal{E}_n.$$  The graph $E_I := (E_I^0, E_I^1, r_I, s_I)$ is defined by 
\begin{align*}
E_I^0 &:= \{v^1 : \text{$v \in E^0$ and $v \neq w$} \} \cup \{ w^1, \ldots, w^n \} \\
E_I^1 &:= \{ e^1 : \text{$e \in E^1$ and $s(e) \neq w$} \} \cup \{ e^1, \ldots, e^n : e \in E^1 \text{ and } s(e) = w\} \\
r_{E_I} (e^{i }) &:= \begin{cases} r(e)^1 & \text{if $e \in E^1$ and $r(e) \neq w$} \\ r(e)^j & \text{if $e \in E^1$ and $r(e) = w$ with $e \in \mathcal{E}_j$}  \end{cases} \\
s_{E_I} (e^i) &:= \begin{cases} s(e)^1 & \text{if $e \in E^1$ and $s(e) \neq w$} \\ w^i & \text{if $e \in E^1$ and $s(e) = w$}.
\end{cases}
\end{align*}
We call $E_I$ the \emph{graph obtained by insplitting $E$ at $w$}, and say $E_I$ is formed by \emph{performing Move (I)} to $E$.
\end{definition}

\begin{proposition}[Cf.~Theorem~5.3 of \cite{BP}] \label{Move-I-ME-prop}
Suppose $E$ is a graph, $w \in E^0$ is a regular vertex that is not a source, and a partition $$r^{-1}(w) = \mathcal{E}_1 \sqcup \mathcal{E}_2 \sqcup \ldots \sqcup \mathcal{E}_n$$ is chosen with the $\mathcal{E}_i$ disjoint nonempty sets.  Let $K$ be any field and define $V := \{ v^1 : v \in E^0 \}$.  Then $VL_K(E_I)V$ is a full corner of $L_K(E_I)$ that is isomorphic to $L_K(E)$ via a graded isomorphism.  Consequently, $L_K(E_I)$ is Morita equivalent to $L_K(E)$.
\end{proposition}

\begin{proof}
For $v \in E^0$ and $e \in E^1$ define $P_v := v^1$ and 
$$S_e := \begin{cases} \ e^1 & \text{if $r(e) \neq w$} \\ \displaystyle \sum_{\{ f \in E^1 : s(f) = w\}} e^1 f^i (f^1)^{*}  & \text{if $r(e) = w$ and $e \in \mathcal{E}_i$.} \end{cases} $$  It is straightforward to verify that $\{ P_v, S_e, S_e^* : v \in E^0, e \in E^1 \}$ is an $E$-family in $L_K(E_I)$, and hence there exists a homomorphism $\phi :  L_K(E) \to L_K(E_I)$ with $\phi (v) = P_v$, $\phi(e) = S_e$, and $\phi(e^*) = S_e^*$ for all $v \in E^0$ and all $e \in E^1$.   Since $P_v$ has degree $0$, $S_e$ has degree $1$, and $S_e^*$ has degree $-1$, we see that $\phi$ is a graded homomorphism.  Therefore, since $\phi(v) = P_v \neq 0$ for all $v \in E_S^0$, the graded uniqueness theorem implies that $\phi$ is injective.  In addition, if we let $V := \{ v^1 : v \in E^0 \}$, then it is straightforward to verify that the corner determined by $V$ is $VL_K(E_I)V = \im \phi$.  Moreover, since $w$ is a regular vertex in $E$, $w^i$ is a regular vertex in $E_I$ for all $1 \leq i \leq n$, and hence $\Sigma (V) = E^0$.  Hence $\im \phi$ is a full corner of $L_K(E_I)$, and by Lemma~\ref{corner-ME-lem} $\im \phi$ is Morita equivalent to $L_K(E)$.
\end{proof}

\begin{definition}[Move (R): Reduction at a Regular Vertex] \label{Move-R-def}
Suppose that $E = (E^0, E^1, r, s)$ is a graph, and let $w \in E^0$ be a regular vertex with the property that $s(r^{-1}(w)) = \{ x \}$, $s^{-1}(w) = \{ f \}$, and $r(f) \neq w$.  The graph $E_R := (E_R^0, E_R^1, r_R, s_R)$ is defined by 
\begin{align*}
E^0_R &:= E^0 \setminus \{ w \} \\
E^1_R &:= E^1 \setminus ( \{ f \} \cup r^{-1}(w) ) \cup \{ e_f : e \in E^1 \text{ and } r(e) =w \} \\
r_R(e) &:= r(e) \text{ if $e \in E^1 \setminus ( \{ f \} \cup r^{-1}(w) )$} \qquad \text{ and } \qquad r_R(e_f) := r(f) \\
s_R(e) &:= s(e) \text{ if $e \in E^1 \setminus ( \{ f \} \cup r^{-1}(w) )$} \qquad \text{ and } \qquad s_R(e_f) := s(e) = x
\end{align*}
We call $E_R$ the \emph{graph obtained by reducing $E$ at $w$}, and say $E_R$ is a \emph{reduction} of $E$ or that $E_I$ is formed by \emph{performing Move (R)} to $E$.
\end{definition}

\begin{proposition}[Cf.~Section~3 of \cite{Sor}] \label{Move-R-ME-prop}
Suppose $E$ is a graph and $w \in E^0$ is a regular vertex with the property that $s(r^{-1}(w)) = \{ x \}$, $s^{-1}(w) = \{ f \}$, and $r(f) \neq w$.  If $K$ is a field, and if we set $V := E^0 \setminus \{ w \}$, then $VL_K(E)V$ is a full corner of $L_K(E)$ that is isomorphic to $L_K(E_R)$.  Consequently, $L_K(E_R)$ is Morita equivalent to $L_K(E)$.
\end{proposition}

\begin{proof}
For each $v \in E_R^0$ define $P_v := v$.  For $e \in E^1 \setminus ( \{ f \} \cup r^{-1}(w) )$ define $S_e := e$ and for $e_f \in E^1_R$ define $S_{e_f} := ef$.  It is straightforward to verify that $\{ P_v, S_e, S_e^* : v \in E_R^0, e \in E_R^1 \}$ is an $E_R$-family in $L_K(E)$, and hence there exists a homomorphism $\phi :  L_K(E_R) \to L_K(E)$ with $\phi (v) = P_v$, $\phi(e) = S_e$, and $\phi(e^*) = S_e^*$ for all $v \in E_R^0$ and all $e \in E_R^1$.   (Note that this homomorphism is not graded.)  If $v \in E^0_R$, then we see that $\phi (v) = P_v \neq 0$.  In addition, if $\alpha = e_1 \ldots e_n$ is a cycle in $E_R$ with no exits, then $\phi(\alpha) = \phi(e_1) \ldots \phi(e_n)$, where each $\phi(e_i)$ is equal to either $e_i$ or $e_i f$.  Hence $\phi (\alpha)$ is a cycle of length greater than or equal to the length of $\alpha$, and in particular $\phi(\alpha)$ is a non-nilpotent element of nonzero degree.  It follows from \cite[Theorem~3.7]{AMMS} that $\phi$ is injective.  In addition, if we let $V := E^0 \setminus \{ w \}$, then it is straightforward to verify that the corner determined by $V$ is $VL_K(E)V = \im \phi$.  Moreover, since $w$ emits exactly one edge $f$ to $r(f) \neq w$, we have $\Sigma (V) = E^0$.  Hence $\im \phi$ is a full corner of $L_K(E)$, and by Lemma~\ref{corner-ME-lem} $\im \phi$ is Morita equivalent to $L_K(E)$.
\end{proof}

\begin{remark}
Note that the homomorphism described in Proposition~\ref{Move-R-ME-prop} is not graded.
\end{remark}

When applied to a graph, the four moves (S), (O), (I), and (R) preserve Morita equivalence of the associated Leavitt path algebra; moreover, this is done in a very concrete way --- for each move we can describe a set of vertices that determines a full corner for implementing the Morita equivalence.  We make this precise in the following theorem.

\begin{definition} \label{move-equivalent-def}
We let $\sim_M$ denote the equivalence relation on $\mathcal{S}$ generated by the four moves (S), (O), (I), and (R).  In other words, if $E, F \in \mathcal{S}$, then $E \sim_M F$ if and only if there exists a finite sequence of graphs $E_0, E_1, \ldots, E_n$ such that $E_0 = E$, $E_n = F$, and for all $1 \leq i \leq n-1$, either $E_i$ is obtained by applying one of the four moves (S), (O), (I), or (R) to $E_{i+1}$, or $E_{i+1}$ is obtained by applying one of the four moves (S), (O), (I), or (R) to $E_{i}$.  If $E$ and $F$ are graphs and $E \sim_M F$, then we say $E$ is \emph{move equivalent} to $F$.
\end{definition}

\begin{theorem} \label{moves-imply-concrete-ME-thm}
If $E$ and $F$ are graphs and $E \sim_M F$, then for any field $K$ it is the case that $L_K(E)$ is Morita equivalent to $L_K(F)$.  Moreover, in this case there exists a finite sequence of graphs $E_0, E_1, \ldots, E_n$  with $E_0 = E$ and $E_n = F$, and there exist subsets of vertices $V_i \subseteq E_i^0$ and $W_i \subseteq E_{i+1}^0$ for each $0 \leq i \leq n-1$ such that 
\begin{itemize}
\item[(a)] $V_i  L_K(E_i) V_i$ is a full corner of $L_K(E_i)$ for $0 \leq i \leq n-1$,
\item[(b)] $W_i L_K(E_{i+1}) W_i$ is a full corner of $L_K(E_{i+1})$ for $0 \leq i \leq n-1$, and
\item[(c)] $V_i L_K(E_i) V_i \cong W_i L_K(E_{i+1}) W_i$ for $0 \le i \leq n-1$.
\end{itemize}
\end{theorem}

\begin{proof}
It is shown in Proposition~\ref{Move-S-ME-prop}, Proposition~\ref{Move-O-ME-prop}, Proposition~\ref{Move-I-ME-prop}, and Proposition~\ref{Move-R-ME-prop} that when the moves (S), (O), (I), and (R) are applied to graphs they each preserve Morita equivalence of the associated Leavitt path algebras, and moreover, if one graph is obtained from another via one of these moves, then the Leavitt path algebra of one graph is isomorphic to a full corner of the other graph, and this full corner is determined by a set of vertices in the graph.
\end{proof}

We make some remarks about the four moves (S), (O), (I), and (R).

\begin{remark}
Move (O) preserves the isomorphism class of the associated Leavitt path algebra, but in general the other three moves preserve only Morita equivalence and not isomorphism class.  Moves (S), (O), and (I) preserve the grading of the associated Leavitt path algebras, but move (R) does not.  
\end{remark}

\begin{remark} \label{move-effects-rem}
Moves (S), (I), and (R) must be performed at regular vertices.  Move (O) is the only move we can perform at a singular vertex, and Move~(O) may be performed at a regular vertex or an infinite emitter, but not at a sink.  When Move~(O) is performed at an infinite emitter, only one of the sets in the partitioned outgoing edges is allowed to be infinite, so the number of infinite emitters in the graph stays the same.  As a result, we see that each of the moves (S), (O), (I), and (R) preserves the number of singular vertices in a graph.
\end{remark}

\begin{corollary} \label{moves-preserve-sing-cor}
If $E$ and $F$ are graphs with a finite number of vertices and $E \sim_M F$, then $|E^0_\textnormal{sing}| = |F^0_\textnormal{sing}|$.
\end{corollary}

%%%%%%%%%%%%%%%%%%%%%%%%%%%%%%%%%%%%%%%%%%%%%%%%%%%%%%
\section{Simple Graphs} \label{simple-graph-sec}
%%%%%%%%%%%%%%%%%%%%%%%%%%%%%%%%%%%%%%%%%%%%%%%%%%%%%%

If $\alpha = e_1e_2 \ldots e_n$ is a cycle, an \emph{exit} for $\alpha$ is an edge $f \in E^1$ such that $s(f) = s(e_i)$ and $f \neq e_i$ for some $i$.  A graph is said to satisfy \emph{Condition~(L)} if every cycle in the graph has an exit.  An \emph{infinite path} in a graph $E$ is an infinite  sequence of edge $\mu := e_1 e_2 \ldots$ with $r(e_i) = s(e_{i+1})$ for all $i \in \N$.  A graph $E$ is called \emph{cofinal} if whenever $\mu := e_1 e_2 \ldots$ is an infinite path in $E$ and $v \in E^0$, then there exists a finite path $\alpha \in E^*$ with $s(\alpha) = v$ and $r(\alpha) = s(e_i)$ for some $i \in \N$.

\begin{definition} \label{simple-graphs-def}
We say that a graph $E$ is \emph{simple} if $E$ satisfies the following three conditions:
\begin{itemize}
\item[(1)] $E$ is cofinal, 
\item[(2)] $E$ satisfies Condition~(L), and 
\item[(3)] whenever $v \in E^0$ and $w \in E^0_\textnormal{sing}$ ,there exists a path $\alpha \in E^*$ with $s(\alpha) = v$ and $r(\alpha) = w$.  
\end{itemize}
We let $\mathcal{S}$ denote the collection of (isomorphism classes) of simple graphs.
\end{definition}

\begin{proposition}
If $E$ is a graph, then the following are equivalent:
\begin{itemize}
\item[(1)] The graph $E$ is simple.
\item[(2)] The Leavitt path algebra $L_K(E)$ is simple for any field $K$.
\item[(3)] The graph $C^*$-algebra $C^*(E)$ is simple.
\end{itemize}
In addition, the collection $\mathcal{S}$ of all simple graphs is closed under the moves (S), (O), (I), (R), and their inverses.
\end{proposition}

\begin{proof}
The equivalence of (1), (2), and (3) is well known.  The equivalence of (1) and (2) is proven in \cite[Proposition~6.18]{Tom10} and \cite[Proposition~6.20]{Tom10}.  The equivalence of (1) and (3) follows from \cite[Theorem~12]{Szy4} and \cite[Proposition~6.20]{Tom10}. 

To establish the final claim, we observe that if $E \in \mathcal{S}$, then $L_K(E)$ is simple.  It follows from Proposition~\ref{Move-S-ME-prop}, Proposition~\ref{Move-O-ME-prop}, Proposition~\ref{Move-I-ME-prop}, and Proposition~\ref{Move-R-ME-prop} that the moves (S), (O), (I), (R), and their inverses preserve Morita equivalence of the Leavitt path algebra and hence produce simple graphs.  Thus $\mathcal{S}$ is closed under the moves (S), (O), (I), (R), and their inverses.  (Alternatively, one can verify directly from the definitions that each of the moves (S), (O), (I), (R), and their inverses preserve cofinality, Condition~(L), and reachability of singular vertices.)
\end{proof}

\begin{proposition} \label{inf-emitter-implies-pi-prop}
If $E$ is a simple graph and $E$ contains an infinite emitter, then $E$ contains a cycle, the graph $C^*$-algebra $C^*(E)$ is purely infinite, and the Leavitt path algebra $L_K(E)$ is purely infinite for any field $K$.
\end{proposition}

\begin{proof}
Let $E$ be a simple graph, and suppose that $v \in E^0$ is an infinite emitter.  Choose $e \in E^1$ with $s(e) = v$.  Since $E$ is simple, there is a path from $r(e)$ to the infinite emitter $v = s(e)$.  Hence $E$ contains a cycle.  It follows from the dichotomy for simple graph $C^*$-algebras that $C^*(E)$ is purely infinite, and it follows from the dichotomy for Leavitt path algebras that $L_K(E)$ is purely infinite for any field $K$.
\end{proof}

%%%%%%%%%%%%%%%%%%%%%%%%%%%%%%%%%%%%%%%%%%%%%%%%%%%%%%
\section{Computing the algebraic $K$-theory of Leavitt path algebras and the topological $K$-theory of graph $C^*$-algebras} \label{compute-K-sec}
%%%%%%%%%%%%%%%%%%%%%%%%%%%%%%%%%%%%%%%%%%%%%%%%%%%%%%

In this section we discuss methods for computing the algebraic $K$-theory of a Leavitt path algebra and the topological $K$-theory of a graph $C^*$-algebra.  We also make some observation about these computations in the case when the graph has a finite number of vertices (or, equivalently, when the Leavitt path algebras and the graph $C^*$-algebra are unital).

Throughout this section, whenever $K$ is a field we let $K^\times := K \setminus \{ 0 \}$ denote the units of $K$.  We often view the set of units as an abelian group $(K^\times, \cdot)$ under the operation of the field multiplication.  In addition, for any group $G$ and any cardinal number $n$, we let $G^n$ denote the direct sum of $n$ copies of $G$.  Note that $n$ is allowed to be infinite.

\begin{proposition} \label{K-theory-comp-prop}
Let $E$ be a graph and decompose the vertices of $E$ as $E^0 = E^0_\textnormal{reg} \sqcup E^0_\textnormal{sing}$, and with respect to this decomposition write the vertex matrix of $E$ as $$A_E = \begin{pmatrix} B_E & C_E \\ * & * \end{pmatrix}$$ where $B_E$ and $C_E$ have entries in $\Z$ and the $*$'s have entries in $\Z \cup \{ \infty \}$.  For each $v \in E^0$, let $\delta_v \in \Z^{E^0}$ denote the vector with $1$ in the $v$\textsuperscript{th} position and $0$'s elsewhere, and for $x \in \Z^{E^0}$ let $[x]$ denote the equivalence class of $x$ in $\coker \left( \begin{pmatrix} B_E^t-I \\ C_E^t\end{pmatrix} : \Z^{E^0_\textnormal{reg}} \to \Z^{E^0} \right)$.
\begin{itemize}
\item[(a)]  The topological $K$-theory of the graph $C^*$-algebra may be calculated as follows:  We have
$$K_0^\textnormal{top} (C^*(E)) \cong \coker \left( \begin{pmatrix} B_E^t-I \\ C_E^t\end{pmatrix} : \Z^{E^0_\textnormal{reg}} \to \Z^{E^0} \right)$$ via an isomorphism that takes $[p_v]_0 \mapsto [\delta_v]$ and takes the positive cone of  $K_0^\textnormal{top} (C^*(E))$ to the cone of $\coker \left( \begin{pmatrix} B_E^t-I \\ C_E^t\end{pmatrix} : \Z^{E^0_\textnormal{reg}} \to \Z^{E^0} \right)$ generated by the elements
$$
\{ [ p_v - \sum_{e \in F} s_es_e^*] : v \in E^0, F \subseteq s^{-1}(v), \text{ and $F$ finite} \},$$ and we have $$K_1^\textnormal{top} (C^*(E)) \cong \ker \left( \begin{pmatrix} B_E^t-I \\ C_E^t\end{pmatrix} : \Z^{E^0_\textnormal{reg}} \to \Z^{E^0} \right).$$ 

\item[(b)] If $K$ is any field, then the algebraic $K$-theory of the Leavitt path algebra $L_K(E)$ may be calculated as follows:  We have
$$K_0^\textnormal{alg} (L_K(E)) \cong \coker \left( \begin{pmatrix} B_E^t-I \\ C_E^t\end{pmatrix} : \Z^{E^0_\textnormal{reg}} \to \Z^{E^0} \right)$$ via an isomorphism that takes $[v]_0 \mapsto [\delta_v]$ and takes the positive cone of  $K_0^\textnormal{alg} (L_{K}(E))$ to the cone of $\coker \left( \begin{pmatrix} B_E^t-I \\ C_E^t\end{pmatrix} : \Z^{E^0_\textnormal{reg}} \to \Z^{E^0} \right)$ generated by the elements
$$
\{ [ v - \sum_{e \in F} e e^*] : v \in E^0, F \subseteq s^{-1}(v), \text{ and $F$ finite} \},$$ and we have 
\begin{align*}
K_1^\textnormal{alg} (L_K(E)) \cong \ker \bigg( &\begin{pmatrix} B_E^t-I \\ C_E^t\end{pmatrix} : \Z^{E^0_\textnormal{reg}} \to \Z^{E^0} \bigg) \\
& \oplus  \coker \left( \begin{pmatrix} B_E^t-I \\ C_E^t\end{pmatrix} : (K_1^\textnormal{alg}(K))^{E^0_\textnormal{reg}} \to (K_1^\textnormal{alg} (K))^{E^0} \right)
\end{align*}
with $(K_1^\textnormal{alg}(K), +) \cong (K^\times, \cdot)$.
Moreover, there is a long exact sequence
$$\xymatrix{ K_n^\textnormal{alg}(K)^{E^0_\textnormal{reg}} \ar[r]^{ \left( \begin{smallmatrix} B_E^t-I \\ C_E^t\end{smallmatrix} \right)} & K_n^\textnormal{alg}(K)^{E^0} \ar[r] & K_n^\textnormal{alg}(L_K(E)) \ar[r] & K_{n-1}^\textnormal{alg}(K)^{E^0_\textnormal{reg}}}
$$
for $n \in \Z$.
\end{itemize}
\end{proposition}

\begin{proof}
The result of Part~(a) follows from \cite[Proposition~3.8]{CET}.  For Part~(b), when $E$ is row-finite the computation of $K_0^\textnormal{alg} (L_K(E))$ and $K_1^\textnormal{alg} (L_K(E))$ is established in \cite[Corollary~7.7]{ABC}, and the long exact sequence follows from \cite[Theorem~7.6]{ABC}.  When $E$ is not row-finite, we apply the desingularization construction introduced in \cite{DT1}.  It follows from \cite[Theorem~5.2]{AbrPino3} that the desingularization preserves the Morita equivalence of the associated Leavitt path algebras, and hence does not change the isomorphism classes of the $K_0^\textnormal{alg}$-group and $K_1^\textnormal{alg}$-group.  The computation above then follows from the row-finite case and \cite[Lemma~2.3]{DT2}.  Moreover, the claim about the positive cone of $K_0^\textnormal{alg}$ is established in \cite[Theorem~4.9]{HLMRT}.  Also, since $K$ is a field, we have $(K_1^\textnormal{alg}(K), +) \cong (K^\times, \cdot)$.
\end{proof}

\begin{remark}\label{r:K-theory-comp-prop}
When we write $K_1^\textnormal{alg}(K)$ additively, the matrix multiplication $\begin{pmatrix} B_E^t-I \\ C_E^t\end{pmatrix} : (K_1^\textnormal{alg}(K))^{E^0_\textnormal{reg}} \to (K_1^\textnormal{alg} (K))^{E^0}$ is the usual matrix multiplication.  If we make the identification $(K_1^\textnormal{alg}(K), +) \cong (K^\times, \cdot)$, we shall then write $\begin{pmatrix} B_E^t-I \\ C_E^t\end{pmatrix} : (K^\times)^{E^0_\textnormal{reg}} \to (K^\times)^{E^0}$, however, it is important the reader keep in mind that the group operation on $K^\times$ is written multiplicatively, not additively.  For example, if we write $\begin{pmatrix} a & b \\ c & d \end{pmatrix} : (K^\times)^2 \to (K^\times)^2$, then this map takes $\begin{pmatrix}x \\ y \end{pmatrix} \in (K^\times)^2$ to the element $\begin{pmatrix} x^ay^b \\ x^cy^c \end{pmatrix} \in (K^\times)^2$.
\end{remark}

Using this $K$-theory computation, we can derive a number of corollaries regarding graphs with finitely many vertices.

\begin{corollary} \label{K-theory-for-finite-vertices-cor}
If $E$ is a graph with finitely many vertices, then there exists $d_1, \ldots, d_k \in \{2, 3, 4, \ldots \}$ and $m \in \N \cup \{ 0 \}$ such that $d_i \, | \, d_{i+1}$ for all $1 \leq i \leq k-1$, the topological $K$-theory of $C^*(E)$ is given by
\begin{align*}
K_0^\textnormal{top} (C^*(E)) &\cong \Z_{d_1} \oplus \ldots \oplus \Z_{d_k} \oplus \Z^{m + |E^0_\textnormal{sing}|} \\
K_1^\textnormal{top} (C^*(E)) &\cong \Z^m,
\end{align*}
and for any field $K$ the algebraic $K$-theory of $L_K(E)$ is given by
\begin{align*}
K_0^\textnormal{alg} (L_K(E)) &\cong \Z_{d_1} \oplus \ldots \oplus \Z_{d_k} \oplus \Z^{m + |E^0_\textnormal{sing}|} \\
K_1^\textnormal{alg} (L_K(E)) &\cong \Z^m \oplus K^{\times} / \langle x^{ d_{1} } : x \in K^{\times} \rangle \oplus \ldots \\
&\qquad \qquad \qquad \ldots \oplus K^{\times} / \langle x^{ d_{k} } : x \in K^{\times} \rangle \oplus (K^\times)^{m + |E^0_\textnormal{sing}|}.
\end{align*}
\end{corollary}

\begin{proof}
Consider the matrix $\begin{pmatrix} B_E^t-I \\ C_E^t\end{pmatrix}$.  The Smith normal form over the integers for this matrix has the form $\begin{pmatrix} D \\ 0 \end{pmatrix}$ where $0$ is the $|E^0_\textnormal{sing}| \times |E^0_\textnormal{reg}|$ matrix of all $0$'s, and $D$ is the $|E^0_\textnormal{reg}| \times |E^0_\textnormal{reg}|$ diagonal matrix $$D = \operatorname{diag} ( 1, \ldots, 1, d_1, \ldots, d_k, 0, \ldots, 0 )$$ with
$n$ values of $1$ along the diagonal for some $n \in \N \cup \{ 0 \}$,  the values $d_1, \ldots, d_k$ are integers that are 2 or larger with $d_i \, | \, d_{i+1}$ for all $1 \leq i \leq k-1$, and $m$ values of $0$ appearing along the diagonal for some $m \in \N \cup \{ 0 \}$.

The topological $K$-theory of $C^*(E)$ is then given by
$$K_0^\textnormal{top} (C^*(E)) \cong \coker \left( \begin{pmatrix} D \\ 0 \end{pmatrix} : \Z^{|E^0_\reg|} \to \Z^{|E^0|} \right) \cong  \Z_{d_1} \oplus \ldots \oplus \Z_{d_k} \oplus \Z^{m + |E^0_\textnormal{sing}|}$$
and
$$K_1^\textnormal{top} (C^*(E)) \cong \ker \left( \begin{pmatrix} D \\ 0 \end{pmatrix} : \Z^{|E^0_\reg|} \to \Z^{|E^0|} \right) \cong   \Z^m.$$
For any field $K$, the algebraic $K$-theory of $L_K(E)$ is given by
$$K_0^\textnormal{alg} (L_K(E)) \cong \coker \left( \begin{pmatrix} D \\ 0 \end{pmatrix} : \Z^{|E^0_\reg|} \to \Z^{|E^0|} \right) \cong  \Z_{d_1} \oplus \ldots \oplus \Z_{d_k} \oplus \Z^{m + |E^0_\textnormal{sing}|}$$
and
\begin{align*}
K_1^\textnormal{alg} (L_{K}(E)) &\cong \ker \left( \begin{pmatrix} D \\ 0 \end{pmatrix} : \Z^{|E^0_\reg|} \to \Z^{|E^0|} \right) \\
& \qquad \qquad \qquad \oplus  \coker \left( \begin{pmatrix} D \\ 0 \end{pmatrix} : (K^\times)^{|E^0_\reg|} \to (K^\times)^{|E^0|} \right) \\\
&\cong \Z^m \oplus K^{\times} / \langle x^{ d_{1} } : x \in K^{\times} \rangle \oplus \ldots \\
&\qquad \qquad \qquad \ldots \oplus K^{\times} / \langle x^{ d_{k} } : x \in K^{\times} \rangle  \oplus (K^\times)^{m + |E^0_\textnormal{sing}|}.
\end{align*}
\end{proof}

\begin{corollary} \label{sing-vertices-SME-inv-cor}
If $E$ is a graph with a finite number of vertices, then $K_0(C^*(E))$ and $K_1(C^*(E))$ are finitely generated abelian groups, and $$|E^0_\textnormal{sing}| = \rank K_0^\textnormal{top}(C^*(E)) - \rank K_1^\textnormal{top}(C^*(E)).$$ 
Consequently, if $E$ and $F$ are graphs with a finite number of vertices, and if $C^*(E)$ is strongly Morita equivalent to $C^*(F)$, then $|E^0_\textnormal{sing}| = |F^0_\textnormal{sing}|$. 
\end{corollary}

\begin{corollary} \label{K0+K1-is-K1+sing-cor}
Let $E$ and $F$ be graphs that each have a finite number of vertices. The following are equivalent:
\begin{itemize}
\item[(1)] $K_0^\textnormal{top} (C^*(E)) \cong K_0^\textnormal{top} (C^*(F))$ and $K_1^\textnormal{top} (C^*(E)) \cong K_1^\textnormal{top} (C^*(F))$
\item[(2)] $K_0^\textnormal{top} (C^*(E)) \cong K_0^\textnormal{top} (C^*(F))$ and $|E^0_\textnormal{sing}| = |F^0_\textnormal{sing}|$. 
\end{itemize}
\end{corollary}

%%%%%%%%%%%%%%%%%%%%%%%%%%%%%%%%%%%%%%%%%%%%%%%%%%%%%%
\section{Comparison of algebraic $K$-theory of Leavitt path algebras and topological $K$-theory of graph $C^*$-algebras} \label{K-theory-comparison-sec}
%%%%%%%%%%%%%%%%%%%%%%%%%%%%%%%%%%%%%%%%%%%%%%%%%%%%%%

In this section we examine the relationship between the topological $K$-theory of a graph $C^*$-algebra and the algebraic $K$-theory of the corresponding Leavitt path algebra for the graph. Interestingly, we find that this relationship depends on the field.  

As in the previous section, if $K$ is a field, we let $K^\times := K \setminus \{ 0 \}$ denote the units of $K$.  Also, for any group $G$ and any cardinal number $n$, we let $G^n$ denote the direct sum of $n$ copies of $G$.

\begin{definition} \label{nfq-group-def}
Let $G$ be an abelian group.  We say $G$ \emph{has no free quotients} if whenever $H$ is a proper subgroup of $G$, then the quotient $G/H$ is not a free abelian group.  (If $G$ is not a group with no free quotients, we say $G$ \emph{has a free quotient} or $G$ \emph{is a group with free quotients}.)
\end{definition}

\begin{proposition}
Let $G$ be an abelian group.  Then the following are equivalent:
\begin{itemize}
\item[(1)] $G$ has no free quotients.
\item[(2)] $G$ is not a direct sum of a free abelian group; i.e., If $H$ and $K$ are abelian groups and $G \cong H \oplus K$, then neither $H$ nor $K$ is a nonzero free abelian group.
\item[(3)] For every free abelian group $F$, it is the case that $\operatorname{Hom}_\Z(G, F) = \{ 0 \}$.
\end{itemize}
\end{proposition}

\begin{proof}
$(1) \implies (3)$.  Suppose that $G$ has no free quotients.  Let $F$ be a free abelian group, and let $\phi : G \to F$ be a group homomorphism.  Then $\im \phi$ is a free abelian group, because subgroups of free abelian groups are free abelian.  Since $G / \ker \phi \cong \im \phi$, and $G$ has no free quotients, it follows that $\ker \phi = G$ and $\phi = 0$.

$(3) \implies (2)$.  Suppose that $H$ and $K$ are abelian groups and that $G \cong H \oplus K$.  Then the homomorphism $\pi_1 : H \oplus K \to H$ given by $\pi_1(h,k) = h$ is surjective, and the homomorphism $\pi_2 : H \oplus K \to K$ given by $\pi_2(h,k) = k$ is surjective.  Hence $(3)$ implies that neither $H$ nor $K$ are nonzero free abelian groups.

$(2) \implies (1)$.  We shall establish the contrapositive.  Suppose that $G$ has a free quotient; that is, there exists a proper subgroup $H$ of $G$ such that $G/H$ is a free abelian group.  The fact that $G/H$ is free implies that the exact sequence $0 \to H \to G \to G/H \to 0$ is split exact.  Hence $G \cong H \oplus G/H$.  Since $H$ is a proper subgroup, $G/H$ is a nonzero free group, and $(2)$ does not hold.
\end{proof}

Abelian groups with no free quotients will be useful for us due to the following lemma.

\begin{lemma} \label{nfq-direct-sums-cancel-prop}
Suppose that $F_1$ and $F_2$ are free abelian groups and that $G_1$ and $G_2$ are abelian groups with no free quotients.  If $F_1 \oplus G_1 \cong F_2 \oplus G_2$, then $F_1 \cong F_2$ and $G_1 \cong G_2$.
\end{lemma}

\begin{proof}
Let $\phi : F_1 \oplus G_1 \to F_2 \oplus G_2$ be an isomorphism.  Consider the group homomorphism $\iota : G_1 \to F_1 \oplus G_1$ given by $\iota(g) := (0,g)$, and the group homomorphism $\pi : F_2 \oplus G_2 \to F_2$ given by $\pi (f,g) = f$.  Since the composition $\pi \circ \phi \circ \iota : G_1 \to F_2$ is a homomorphism, $F_2$ is a free abelian group, and $G_1$ is an abelian group with no free quotients, it follows that $\pi \circ \phi \circ \iota = 0$.  Thus, $\im ( \phi \circ \iota) \subseteq 0 \oplus G_{2}$ which implies that $\phi (0 \oplus G_1) \subseteq 0 \oplus G_2$.  A similar argument using $\phi^{-1}$ in place of $\phi$ shows that $\phi^{-1} (0 \oplus G_2) \subseteq 0 \oplus G_1$.  Hence $\phi (0 \oplus G_1) = 0 \oplus G_2$, and we may conclude that $G_1 \cong G_2$.  In addition, $F_1 \cong (F_1 \oplus G_1) / (0 \oplus G_1) \cong (F_2 \oplus G_2) / \phi(0 \oplus G_1) = (F_2 \oplus G_2) / (0 \oplus G_2) \cong F_2$.
\end{proof}

\begin{proposition} \label{nfq-closed-sums-quotients-prop}
The class of abelian groups that have no free quotients is closed under (possibly infinite) direct sums and quotients.
\end{proposition}

\begin{proof}
It is easy to see that the class is closed under taking quotients, since quotients of quotients are quotients.

To see that the class is closed under direct sums, let  $\{ G_i \}_{i \in I}$ be a collection of abelian groups with no free quotients.  To show $\bigoplus_{i \in I} G_i$ has no free quotients, it suffices to show that if $F$ is a free group and $\phi : \bigoplus_{i \in I} G_i \to F$ is a group homomorphism, then $\phi$ is the zero map.  To this end, for each $i \in I$ let $h_i : G_i \to \bigoplus_{i \in I} G_i$ be the homomorphism sending an element $a \in G_i$ to the element $h_i(a) \in  \bigoplus_{i \in I} G_i$ that has $a$ in the $i$\textsuperscript{th} coordinate and $0$ elsewhere.  For any $x \in \bigoplus_{i \in I} G_i$ we may write $x = \sum_{i \in I} h_i (a_i)$ where $a_i \in G_i$ for all $i \in I$ and all but finitely many of the $a_i$ are zero.  The fact that $G_{i}$ has no free quotients implies that the homomorphism $\phi \circ h_i : G_i \to F$ is zero for all $i \in I$.  Thus $\phi (x) = \phi (  \sum_{i \in I} h_i (a_i) ) = \sum_{i \in I} \phi ( h_i (a_i)) = 0$.  Since $x \in \bigoplus_{i \in I} G_i$ was arbitrary, the map $\phi : \bigoplus_{i \in I} G_i \to F$ is the zero homomorphism, and $\bigoplus_{i \in I} G_i$ has no free quotients.
\end{proof}

We now consider several sufficient conditions for an abelian group to have no free quotients.

\begin{definition}
Let $(G,+)$ be an abelian group.  We make the following definitions:
\begin{itemize}
\item[(a)] We say $G$ is \emph{divisible} if for every $y \in G$ and for every $n \in \N$ there exists $x \in G$ such that $nx = y$.   
\item[(b)] If $p$ is a prime, we say $G$ is \emph{$p$-divisible} if for every $y \in G$ and for every $k \in \N$, there exists $x \in G$ such that $p^k x = y$.  
\item[(c)] We say that $G$ is \emph{weakly divisible} if for every $y \in G$ and for every $N \in \N$ there exist $n \geq N$ and $x \in G$ such that $nx = y$. 
\end{itemize}
\end{definition}

\begin{remark}
The terms ``divisible" and ``$p$-divisible" are established terminology in group theory, while ``weakly divisible" is a term we introduce here.  It is fairly easy to see that an abelian group is divisible if and only if it is $p$-divisible for every prime $p$.  In addition, if $G$ is $p$-divisible for some prime $p$, then $G$ is weakly divisible.
\end{remark}

\begin{proposition} \label{weakly-div-implies-nfq}
If $G$ is a weakly divisible abelian group, then $G$ has no free quotients.
\end{proposition}

\begin{proof}
To begin, we see that in a free abelian group the only element that is divisible by arbitrarily large $n \in N$ is the zero element.  Therefore, no nontrivial weakly divisible abelian group is free abelian.  Likewise, since any quotient of a weakly divisible group is weakly divisible, we see that no nontrivial quotient of a weakly divisible group is free abelian.  Hence a weakly divisible abelian group has no free quotients.
\end{proof}

\begin{remark}
Note that if $(G, \cdot)$ is an abelian group with the group operation written multiplicatively instead of additively, then $G$ is divisible if every element has an $n$\textsuperscript{th}-root in $G$ for all $n \in \N$, $G$ is $p$-divisible if every element has a $(p^k)$\textsuperscript{th}-root in $G$ for all $k \in \N$, and $G$ is weakly divisible if every element has an $n$\textsuperscript{th}-root in $G$ for arbitrarily large $n$.
\end{remark}

\begin{definition} \label{nfq-field-def}
If $K$ is a field, we say that $K$ \emph{has no free quotients} if the group of units $(K^\times, \cdot)$ is an abelian group that has no free quotients.  (If $K$ is not a field with no free quotients, we say $K$ \emph{has a free quotient}.)

\end{definition}

\begin{proposition} \label{nfg-field-examples}
The class of fields with no free quotients contains the following:

\begin{itemize}
\item All fields $K$ such that $(K^\times, \cdot)$ is a torsion group.
\item All fields $K$ such that $(K^\times, \cdot)$ is weakly divisible.  (This includes all fields $K$ such that $(K^\times, \cdot)$ is $p$-divisible for some prime $p$, and it also includes all fields $K$ such that $(K^\times, \cdot)$ is divisible.)
\item All algebraically closed fields.
\item All fields that are perfect with characteristic $p >0$.
\item All finite fields.
\item The field $\C$ of complex numbers.
\item The field $\R$ of real numbers.
\end{itemize}
\end{proposition}

\begin{proof}
If $K$ is a field and $(K^\times, \cdot)$ is a torsion group, then any nonzero quotient of $K^\times$ has torsion and is not a free group, so that $K$ has no free quotients.  If $K$ is a field and $(K^\times, \cdot)$ is weakly divisible, then it follows from Proposition~\ref{weakly-div-implies-nfq} that $K$ has no free quotients.  If $K$ is algebraically closed, then for any $y \in K^\times$ and any $n \in \N$, the equation $x^n = y$ in the variable $x$ has a nonzero root in $K$, implying $(K^\times, \cdot)$ is divisible and hence also weakly divisible.  If $K$ is perfect with characteristic $p > 0$, then the Frobenius endomorphism $x \mapsto x^p$ is an automorphism, and every element of $K$ has a $p$\textsuperscript{th}-root, which implies that $(K^\times, \cdot)$ is $p$-divisible, and hence weakly divisible.  If $K$ is a finite field, then $K$ is perfect with characteristic $p > 0$, and thus covered by the prior case.  The field $\C$ of complex numbers is algebraically closed, and thus covered by a prior case.  Since every real number has a real cube root, the field $\R$ of real numbers  is $3$-divisible, and thus weakly divisible, and covered by a prior case.
\end{proof}

The following proposition shows us that $\Q$ is not a field with no free quotients.

\begin{proposition} \label{Q-times-Z-inf-prop}
If $\Q$ is the field of rational numbers, then $(\Q^\times, \cdot) \cong (\Z_{2} \oplus \Z \oplus \Z \oplus \ldots, +)$.  
\end{proposition}

\begin{proof}
List the positive prime integers in increasing order as $p_1, p_2, p_3, \ldots$.  (So, in particular, $p_1 = 2, p_2 = 3, p_3 = 5, \ldots$.)  Then it is straightforward to verify that the map from $\Z_{2} \oplus \Z \oplus \Z \oplus \ldots \to \Q^\times$ given by $$([m] , n_1, n_2, n_3, \ldots ) \mapsto (-1)^{m} p_1^{n_1} p_2^{n_2} p_3^{n_3} \ldots$$ is a group isomorphism from $(\Z_{2}\oplus \Z \oplus \Z \oplus \ldots, +)$ onto $(\Q^\times, \cdot)$.
\end{proof}

\begin{remark}
Proposition~\ref{Q-times-Z-inf-prop} shows that $\Q$ is not a field with no free quotients.  Since $\Q$ is a perfect field of characteristic 0, we see that the condition $p>0$ cannot be removed in the fourth bulleted point of Proposition~\ref{nfg-field-examples}.
\end{remark}

\begin{theorem} \label{LPA-K-implies-C-star-K-thm}
Let $E$ and $F$ be graphs.
\begin{itemize}
\item[(1)] If $K_0^\textnormal{alg}(L_K(E)) \cong K_0^\textnormal{alg}(L_K(F))$, then $K_0^\textnormal{top} (C^*(E)) \cong K_0^\textnormal{top} (C^*(F))$.  In addition, if the isomorphism from $K_0^\textnormal{alg}(L_K(E))$ to $K_0^\textnormal{alg}(L_K(F))$ is an order isomorphism, then there is an order isomorphism from $K_0^\textnormal{top} (C^*(E))$ onto $K_0^\textnormal{top} (C^*(F))$.  
\item[(2)] If $K$ is a field with no free quotients (see Definition~\ref{nfq-group-def} and Definition~\ref{nfq-field-def}), and if $K_1^\textnormal{alg}(L_K(E)) \cong K_1^\textnormal{alg}(L_K(F))$, then $K_1^\textnormal{top} (C^*(E)) \cong K_1^\textnormal{top} (C^*(F))$.
\end{itemize}
\end{theorem}

\begin{proof}
Part~(1) follows from Proposition~\ref{K-theory-comp-prop}, which shows that for any graph $E$ the group $K_0^\textnormal{top}(C^*(E))$ is order isomorphic to $K_0^\textnormal{alg}(L_K(E))$.

For Part~(2), we use Proposition~\ref{K-theory-comp-prop}(b) to write
\begin{align*}
K_1^\textnormal{alg} (L_K(E)) \cong \ker \bigg( &\begin{pmatrix} B_E^t-I \\ C_E^t\end{pmatrix} : \Z^{E^0_\textnormal{reg}} \to \Z^{E^0} \bigg) \\
& \oplus  \coker \left( \begin{pmatrix} B_E^t-I \\ C_E^t\end{pmatrix} : (K^\times)^{E^0_\textnormal{reg}} \to (K^\times)^{E^0} \right)
\end{align*}
and
\begin{align*}
K_1^\textnormal{alg} (L_K(F)) \cong \ker \bigg( &\begin{pmatrix} B_F^t-I \\ C_F^t\end{pmatrix} : \Z^{F^0_\textnormal{reg}} \to \Z^{F^0} \bigg) \\
& \oplus  \coker \left( \begin{pmatrix} B_F^t-I \\ C_F^t\end{pmatrix} : (K^\times)^{F^0_\textnormal{reg}} \to (K^\times)^{F^0} \right).
\end{align*}
Since $\Z^{E^0_\textnormal{reg}}$ and $\Z^{F^0_\textnormal{reg}}$ are free abelian groups, it follows that the subgroups $$\ker \bigg( \begin{pmatrix} B_E^t-I \\ C_E^t\end{pmatrix} : \Z^{E^0_\textnormal{reg}} \to \Z^{E^0} \bigg) \quad \text{ and } \quad \ker \bigg( \begin{pmatrix} B_F^t-I \\ C_F^t\end{pmatrix} : \Z^{F^0_\textnormal{reg}} \to \Z^{F^0} \bigg)$$ are free abelian groups.  In addition, since $K$ is a field with no free quotients, $K^\times$ is a group with no free quotients, and it follows from Proposition~\ref{nfq-closed-sums-quotients-prop} that the direct sums $(K^\times)^{E^0}$ and $(K^\times)^{F^0}$ are abelian groups with no free quotients, and that the quotients 
$$\coker \bigg( \begin{pmatrix} B_E^t-I \\ C_E^t\end{pmatrix} : (K^\times)^{E^0_\textnormal{reg}} \to (K^\times)^{E^0} \bigg)$$ and $$\coker \bigg( \begin{pmatrix} B_F^t-I \\ C_F^t\end{pmatrix} : (K^\times)^{F^0_\textnormal{reg}} \to (K^\times)^{F^0} \bigg)$$ are abelian groups with no free quotients.  Since $K_1^{\textnormal{alg}} (L_K(E)) \cong K_1^{\textnormal{alg}}(L_K(F))$, it follows from Lemma~\ref{nfq-direct-sums-cancel-prop} that $$\ker \bigg( \begin{pmatrix} B_E^t-I \\ C_E^t\end{pmatrix} : \Z^{E^0_\textnormal{reg}} \to \Z^{E^0} \bigg) \cong \ker \bigg( \begin{pmatrix} B_F^t-I \\ C_F^t\end{pmatrix} : \Z^{F^0_\textnormal{reg}} \to \Z^{F^0} \bigg).$$  It then follows from Proposition~\ref{K-theory-comp-prop} that $K_1^\textnormal{top} (C^*(E)) \cong K_1^\textnormal{top} (C^*(F))$.
\end{proof}

In the next corollary we use Proposition~\ref{K-theory-comp-prop} and Theorem~\ref{LPA-K-implies-C-star-K-thm} allow us to deduce relationships among the algebraic $K$-groups of Leavitt path algebras, the topological $K$-groups of graph $C^*$-algebras, and the number of singular vertices in the graphs (cf. Corollary~\ref{sing-vertices-SME-inv-cor} and Corollary~\ref{K0+K1-is-K1+sing-cor}).

\begin{corollary} \label{K1-tell-number-singular-cor}
Let $E$ and $F$ be graphs that each have a finite number of vertices, let $K$ be a field, and consider the following statements:
\begin{itemize}
\item[(1)] $K_0^\textnormal{top} (C^*(E)) \cong K_0^\textnormal{top} (C^*(F))$ and $K_1^\textnormal{top} (C^*(E)) \cong K_1^\textnormal{top} (C^*(F))$.
\item[(2)] $K_0^\textnormal{alg} (L_K(E)) \cong K_0^\textnormal{alg} (L_K(F))$ and $|E^0_\textnormal{sing}| = |F^0_\textnormal{sing}|$. 
\item[(3)] $K_0^\textnormal{alg} (L_K(E)) \cong K_0^\textnormal{alg} (L_K(F))$ and $K_1^\textnormal{alg} (L_K(E)) \cong K_1^\textnormal{alg} (L_K(F))$.
\end{itemize}
In general, we have the following implications:
$$(1) \iff (2) \implies (3).$$
If, in addition, $K$ is a field with no free quotients, then the implication $(3) \implies (2)$ also holds, so that all three statements are equivalent when $K$ is a field with no free quotients.
\end{corollary}

\begin{proof}
That $(1) \iff (2)$ follows from Corollary~\ref{K0+K1-is-K1+sing-cor} and the fact that $K_0^\textnormal{top}(C^*(E)) \cong K_0^\textnormal{alg}(L_K(E))$ by Proposition~\ref{K-theory-comp-prop}.

To deduce $(2) \implies (3)$, we use Corollary~\ref{K-theory-for-finite-vertices-cor} to obtain
\begin{align*}
K_0^\textnormal{alg} (L_K(E)) &\cong \Z_{d_1} \oplus \ldots \oplus \Z_{d_k} \oplus \Z^{m_1 + |E^0_\textnormal{sing}|} \\
K_1^\textnormal{alg} (L_K(E)) &\cong \Z^{m_{1}} \oplus K^{\times} / \langle x^{ d_{1} } : x \in K^{\times} \rangle \oplus \ldots \\
& \qquad \qquad \qquad \ldots \oplus K^{\times} / \langle x^{ d_{k} } : x \in K^{\times} \rangle \oplus (K^\times)^{m_{1} + |E^0_\textnormal{sing}|}
\end{align*}
for some $m_1 \in \N \cup \{ 0 \}$ and with $d_i \, | \, d_{i+1}$ for $1 \leq i \leq k-1$.  Corollary~\ref{K-theory-for-finite-vertices-cor} also shows that 
\begin{align*}
K_0^\textnormal{alg} (L_K(F)) &\cong \Z_{e_1} \oplus \ldots \oplus \Z_{e_\ell} \oplus \Z^{m_2 + |F^0_\textnormal{sing}|} \\
K_1^\textnormal{alg} (L_K(F)) &\cong \Z^{m_{2}} \oplus K^{\times} / \langle x^{ e_{1} } : x \in K^{\times} \rangle \oplus \ldots \\
& \qquad \qquad \qquad \ldots \oplus K^{\times} / \langle x^{ e_{\ell} } : x \in K^{\times} \rangle  \oplus (K^\times)^{m_{2} + |F^0_\textnormal{sing}|}.
\end{align*}
for some $m_2 \in \N \cup \{ 0 \}$ and with $e_i \, | \, e_{i+1}$ for $1 \leq i \leq l-1$.  Since we have $K_0^\textnormal{alg} (L_K(E)) \cong K_0^\textnormal{alg} (L_K(F))$ by hypothesis, it follows that $$m_1 + |E^0_\textnormal{sing}| = \rank K_0^\textnormal{alg} (L_K(E)) = \rank K_0^\textnormal{alg} (L_K(F)) = m_2 + |F^0_\textnormal{sing}|,$$ and the fact that $|E^0_\textnormal{sing}| = |F^0_\textnormal{sing}|$ then implies that $m_1 = m_2$.  Since $( d_{1} , d_{2} , \dots, d_{k} )$ and $( e_{1} , e_{2}, \dots, e_{\ell} )$ are the invariant factors of $K_0^\textnormal{alg} (L_K(E))$ and $K_0^\textnormal{alg} (L_K(F))$, respectively, we have that $\ell = k$ and $d_{i} = e_{i}$.  It follows that $K_1^\textnormal{alg} (L_K(E)) \cong K_1^\textnormal{alg} (L_K(F))$.

If, in addition, $K$ is a field with no free quotients, then Theorem~\ref{LPA-K-implies-C-star-K-thm} shows that $(3) \implies (1)$, and since $(1) \iff (2)$ from above, we have $(3) \implies (2)$.
\end{proof}

\begin{remark}
The implication $(3) \implies (2)$ of Corollary~\ref{K1-tell-number-singular-cor} can fail for general fields.  In particular, we will later show in Example~\ref{K-theory-counter-ex} and Example~\ref{K-theory-not-determined-ex} that when $K = \Q$ (and hence $K$ has free quotients) it is possible for $(3)$ to hold while $(2)$ and $(1)$ do not.
\end{remark}

We thank Enrique Pardo for bringing our attention to the following and pointing out that it can be used to strengthen certain results from an earlier version of this paper.

\begin{lemma} \label{l:fieldextension}
Let $K$ be a field, and let $K'$ be a field extension of $K$.  Then $K' \otimes_{K} {M}_{\infty} ( L_{K}(E) ) \cong {M}_{\infty} ( L_{K'}(E) )$ (as $K$-algebras).  Consequently, if $L_{K}(E)$ and $L_{K} (F)$ are Morita equivalent, then $L_{K'}(E)$ and $L_{K'} (F)$ are Morita equivalent. 
\end{lemma}

\begin{proof}
Let $SE$ be the stabilized graph of $E$ formed by attaching an infinite head to each vertex of $E$ (see \cite[Definition~9.1 and Definition~9.4]{AT1}.  By \cite[Proposition~9.8]{AT1}, we have ${M}_{\infty} ( L_{K'} (E) ) \cong L_{K'} (SE)$ (as $K'$-algebras) and ${M}_{\infty} ( L_{K} (E) ) \cong L_{K} (SE)$ (as $K$-algebras).  By \cite[Theorem~8.1]{TomCoeffRing}, $K' \otimes_{K}  L_{K}(SE) \cong L_{K'} (SE)$ as ($K'$-algebras).  Hence $$K' \otimes_{K} {M}_{\infty} ( L_{K}(E) ) \cong K' \otimes_{K}  L_{K}(SE) \cong L_{K'} (SE) \cong {M}_{\infty} (L_{K'}(E))$$ where all the isomorphisms are as $K$-algebras.

In addition, if $L_{K}(E)$ and $L_{K} (F)$ are Morita equivalent, then by \cite[Corollary~9.11]{AT1}, we have that ${M}_{\infty} (L_{K} (E) ) \cong {M}_{\infty} (L_{K}(F) )$ (as rings).  Therefore, using the result from the previous paragraph we have
\begin{align*}
{M}_{\infty} (L_{K'}(E) ) \cong K' \otimes_{K} {M}_{\infty} ( L_{K}(E) ) \cong K' \otimes_{K} {M}_{\infty} ( L_{K}(F) ) \cong {M}_{\infty} (L_{K'}(F) )
\end{align*}
so that ${M}_{\infty} (L_{K'}(E) ) \cong {M}_{\infty} (L_{K'}(F) )$ as rings.  It follows from \cite[Corollary~9.11]{AT1} that $L_{K'} (E)$ is Morita equivalent to $L_{K'}(F)$.
\end{proof}

\begin{corollary} \label{same-number-singular-cor}
If $E$ and $F$ are graphs with a finite number of vertices, $K$ is a field, and $L_K(E)$ is Morita equivalent to $L_K(F)$, then $|E^0_\textnormal{sing}| = |F^0_\textnormal{sing}|$.
\end{corollary}

\begin{proof}
Let $\overline{K}$ be the algebraic closure of $K$.  Since $L_K(E)$ is Morita equivalent to $L_K(F)$, Lemma~\ref{l:fieldextension} implies that $L_{\overline{K}}(E)$ is Morita equivalent to $L_{\overline{K}}(F)$. Because algebraic $K$-theory is a Morita equivalence invariant, we have $K_n^\textnormal{alg} (L_{\overline{K}}(E)) \cong K_n^\textnormal{alg} (L_{\overline{K}}(F))$ for $n  = 0,1$.  Since $\overline{K}$ is algebraically closed, Proposition~\ref{nfg-field-examples} implies that $\overline{K}$ is a field with no free quotients.  The result then follows from $(3) \implies (2)$ of  Corollary~\ref{K1-tell-number-singular-cor}. 
\end{proof}

In \cite[Theorem~8.6 and Corollary~9.16]{AT1} the Isomorphism Conjecture and Morita Equivalence Conjecture are established for Leavitt path algebras over the field $\C$ of complex numbers.  Theorem~\ref{LPA-K-implies-C-star-K-thm} now allows us to verify these conjectures for Leavitt path algebras over arbitrary fields.

\begin{proposition} \label{Iso-Mor-Conj-over-K-prop}
Let $E$ and $F$ be simple graphs, and let $K$ be a field.
\begin{itemize}
\item[(1)] If $L_K(E) \cong L_K(F)$ (as rings), then $C^*(E) \cong C^*(F)$ (as $*$-algebras).
\item[(2)] If $L_K(E)$ is Morita equivalent to $L_K(F)$, then $C^*(E)$ is strongly Morita equivalent to $C^*(F)$.
\end{itemize}
\end{proposition}

\begin{proof}
Since $E$ and $F$ are simple graphs, the dichotomy for Leavitt path algebras implies that $L_K(E)$ and $L_K(F)$ are either both ultramatricial or both purely infinite.  If $L_K(E)$ and $L_K(F)$ are ultramatricial, then since Proposition~\ref{K-theory-comp-prop} implies that $K_0^\textnormal{alg} (L_K(E)) \cong K_0^\textnormal{top}(C^*(E))$ as scaled ordered groups for any graph $E$, the results in (1) and (2) follow from Elliott's classification theorem \cite{Ell2}.  

Suppose $L_K(E)$ and $L_K(F)$ are both purely infinite.  Proposition~\ref{K-theory-comp-prop} shows that for any graph $E$ one has $K_0^\textnormal{alg} (L_K(E)) \cong K_0^\textnormal{top} (C^*(E))$.  Thus, if $L_K(E)$ and $L_K(F)$ are Morita equivalent, it follows that $K_0^\textnormal{alg} (L_K(E)) \cong K_0^\textnormal{alg} (L_K(F))$, and hence $K_0^\textnormal{top} (C^*(E)) \cong K_0^\textnormal{top} (C^*(F))$.  Moreover, $L_K(E) \cong L_K(F)$ implies that $L_K(E)$ and $L_K(F)$ are either both unital or both nonunital.  The Leavitt path algebras $L_K(E)$ and $L_K(F)$ are both unital if and only if the graph $C^*$-algebra $C^*(E)$ and $C^*(F)$ are both unital (which occurs precisely when $E^0$ is a finite set), and in this case, Proposition~\ref{K-theory-comp-prop} shows that  $K_0^\textnormal{alg} (L_K(E)) \cong K_0^\textnormal{top} (C^*(E))$ via an isomorphism taking the class of the identity of $C^*(E)$ to the class of the identity of $C^*(F)$.

Furthermore, if $L_K(E)$ is Morita equivalent to $L_K(F)$, let $\overline{K}$ be the algebraic closure of $K$.  Lemma~\ref{l:fieldextension} implies $L_{\overline{K}}(E)$ is Morita equivalent to $L_{\overline{K}}(F)$, so that $K_1^\textnormal{alg} (L_{\overline{K}}(E)) \cong K_1^\textnormal{alg} (L_{\overline{K}}(F))$.  Because $\overline{K}$ is algebraically closed, Proposition~\ref{nfg-field-examples} shows that $\overline{K}$ is a field with no free quotients, and hence Theorem~\ref{LPA-K-implies-C-star-K-thm} implies that $K_1^\textnormal{top} (C^*(E)) \cong K_1^\textnormal{top} (C^*(F))$.

Thus the previous paragraphs yield two conclusions: First, if $L_K(E) \cong L_K(F)$, then we have $K_0^\textnormal{top} (C^*(E)) \cong K_0^\textnormal{top} (C^*(F))$ via an isomorphism that may be chosen to take the class of the unit to the class of the unit when $C^*(E)$ (and equivalently $C^*(F)$) is unital, and $K_1^\textnormal{top} (C^*(E)) \cong K_1^\textnormal{top} (C^*(F))$.  Second, if $L_K(E)$ is Morita equivalent to $L_K(F)$, then $K_0^\textnormal{top} (C^*(E)) \cong K_0^\textnormal{top} (C^*(F))$ and $K_1^\textnormal{top} (C^*(E)) \cong K_1^\textnormal{top} (C^*(F))$.  Since $L_K(E)$ and $L_K(F)$ are purely infinite and simple, (1) and (2) in the statement of the proposition follow from the Kirchberg-Phillips classification theorem.
\end{proof}

%%%%%%%%%%%%%%%%%%%%%%%%%%%%%%%%%%%%%%%%%%%%%%%%%%%%%%
\section{Classification of unital simple Leavitt path algebras of infinite graphs} \label{class-thm-sec}
%%%%%%%%%%%%%%%%%%%%%%%%%%%%%%%%%%%%%%%%%%%%%%%%%%%%%%

In this section we discuss classification of unital simple Leavitt path algebras of infinite graphs.

\begin{theorem} \label{class-unital-inf-LPA-thm}
Let $E$ and $F$ be simple graphs (see Definition~\ref{simple-graphs-def}) that each have a finite number of vertices and an infinite number of edges.  If $K$ is a field with no free quotients (see Definition~\ref{nfq-group-def} and Definition~\ref{nfq-field-def}), then the following are equivalent:
\begin{itemize}
\item[(1)] $L_K(E)$ is Morita equivalent to $L_K(F)$.
\item[(2)] $K_0^\textnormal{alg} (L_K(E)) \cong K_0^\textnormal{alg} (L_K(F))$ and $K_1^\textnormal{alg} (L_K(E)) \cong K_1^\textnormal{alg} (L_K(F))$.
\item[(3)] $K_0^\textnormal{alg} (L_K(E)) \cong K_0^\textnormal{alg} (L_K(F))$ and $|E^0_\textnormal{sing}| = |F^0_\textnormal{sing}|$.
\item[(4)] $K_0^\textnormal{top} (C^*(E)) \cong K_0^\textnormal{top} (C^*(F))$ and $K_1^\textnormal{top} (C^*(E)) \cong K_1^\textnormal{top} (C^*(F))$.
\item[(5)] $C^*(E)$ is strongly Morita equivalent to $C^*(F)$.
\item[(6)] $E \sim_M F$ (see Definition~\ref{move-equivalent-def}).
\item[(7)] There exists a finite sequence of simple graphs $E_0, E_1, \ldots, E_n$ and nonempty subsets of vertices $V_i \subseteq E_i^0$ and $W_i \subseteq E_{i+1}^0$ for each $0 \leq i \leq n-1$ such that $E_0 = E$, $E_n = F$, and $V_i L_K(E_i) V_i \cong W_i L_K(E_{i+1}) W_i$ for all $0 \leq i \leq n-1$.
\end{itemize}
\end{theorem}

\begin{proof}
Since $E$ and $F$ are graphs that each have an infinite number of edges and a finite number of vertices, it follows that each of $E$ and $F$ have an infinite emitter.  Proposition~\ref{inf-emitter-implies-pi-prop} implies that $C^*(E)$ and $C^*(F)$ are purely infinite simple $C^*$-algebras, and that $L_K(E)$ and $L_K(F)$ are purely infinite simple rings.

$(1) \implies (2)$.  If $L_K(E)$ is Morita equivalent to $L_K(F)$, then since algebraic $K$-theory is a Morita equivalence invariant, $(2)$ follows.

$(2) \implies (3)$ and $(3) \implies (4)$.  These implications follow from Corollary~\ref{K1-tell-number-singular-cor}. (We mention that we need the hypothesis that $K$ is a field with no free quotients to obtain the implication $(2) \implies (3)$.)

$(4) \implies (5)$ and $(5) \implies (6)$. These implications follow from \cite[Theorem~4.8]{Sor}.  

$(6) \implies (7)$. This follows from Theorem~\ref{moves-imply-concrete-ME-thm}.

$(7) \implies (1)$. Since each $E_i$ is a simple graph, each Leavitt path algebra $L_K(E_i)$ is a simple algebra.  Thus any nonzero corner of $L_K(E_i)$ is full, and $L_K(E_i)$ is Morita equivalent to $L_K(E_{i+1})$ for all $1 \leq i \leq n-1$.  Consequently, $L_K(E)$ is Morita equivalent to $L_K(F)$.
\end{proof}

\begin{remark}
We point out that our proof of Theorem~\ref{class-unital-inf-LPA-thm} did not require the Kirchberg-Phillips classification theorem, which is a deep result with a lengthy proof.  Instead we established all the implications using results from this paper and from \cite[Theorem~4.8]{Sor}.  (A self-contained proof of \cite[Theorem~4.8]{Sor} appears in \cite[Section~9]{Sor}.)
\end{remark}

\begin{corollary} \label{no-free-quot-ME-implies-all-ME-cor}
Let $E$ and $F$ be simple graphs that each have a finite number of vertices and an infinite number of edges.  If there exists a field $K$ such that $K$ has no free quotients and $L_{K}(E)$ is Morita equivalent to $L_{K}(F)$, then $L_{K'}(E)$ is Morita equivalent to $L_{K'}(F)$ for every field $K'$.
\end{corollary}

\begin{proof}
Since $K$ is a field with no free quotients and $L_{K}(E)$ is Morita equivalent to $L_{K}(F)$, it follows from Theorem~\ref{class-unital-inf-LPA-thm} that $E \sim_M F$.  Therefore, Theorem~\ref{moves-imply-concrete-ME-thm} implies that $L_{K'}(E)$ is Morita equivalent to $L_{K'}(F)$ for every field $K'$.
\end{proof}

Theorem~\ref{class-unital-inf-LPA-thm} and Corollary~\ref{no-free-quot-ME-implies-all-ME-cor} both have the hypothesis that $K$ is a field with no free quotients.  However, from these two results we obtain the following corollary, which gives a classification for unital Leavitt path algebras of infinite graphs over arbitrary fields.

\begin{theorem} \label{field-does-not-matter-thm}
Let $E$ and $F$ be simple graphs that each have a finite number of vertices and an infinite number of edges.  If $K$ is any field, then the following are equivalent:
\begin{itemize}
\item[(a)] $L_K(E)$ is Morita equivalent to $L_K(F)$.
\item[(b)] $K_0^\textnormal{alg} (L_K(E)) \cong K_0^\textnormal{alg} (L_K(F))$ and $|E^0_\textnormal{sing}| = |F^0_\textnormal{sing}|$.
\item[(c)] $L_{\overline{K}} (E)$ is Morita equivalent to $L_{\overline{K}}(F)$, where $\overline{K}$ denotes the algebraic closure of $K$.
\item[(d)] $L_{K'} (E)$ is Morita equivalent to $L_{K'}(F)$ for every field $K'$.
\item[(e)] $K_0^\textnormal{top} (C^*(E)) \cong K_0^\textnormal{top} (C^*(F))$ and $K_1^\textnormal{top} (C^*(E)) \cong K_1^\textnormal{top} (C^*(F))$.
\item[(f)] $C^*(E)$ is strongly Morita equivalent to $C^*(F)$.
\item[(g)] $E \sim_M F$ (see Definition~\ref{move-equivalent-def}).
\item[(h)] There exists a finite sequence of simple graphs $E_0, E_1, \ldots, E_n$ and nonempty subsets of vertices $V_i \subseteq E_i^0$ and $W_i \subseteq E_{i+1}^0$ for each $0 \leq i \leq n-1$ such that $E_0 = E$, $E_n = F$, and $V_i L_K(E_i) V_i \cong W_i L_K(E_{i+1}) W_i$ for all $0 \leq i \leq n-1$.
\end{itemize}
\end{theorem}

\begin{proof}
The implication $(a) \implies (c)$ follows from Lemma~\ref{l:fieldextension}.  The implication $(c) \implies (d)$ follows from Corollary~\ref{no-free-quot-ME-implies-all-ME-cor} and the fact that $\overline{K}$ is a field with no free quotients.  The implication $(d) \implies (a)$ is trivial.  The equivalences $(c) \iff (e) \iff (f) \iff (g) \iff (h)$ follow from Theorem~\ref{class-unital-inf-LPA-thm} and the fact that $\overline{K}$ is a field with no free quotients.  The equivalence $(b) \iff (e)$ follows from Corollary~\ref{K1-tell-number-singular-cor}.
\end{proof}

\begin{remark}
We mention that the implication $(f) \implies (d)$ of Theorem~\ref{field-does-not-matter-thm} gives a partial converse  to the implication in Proposition~\ref{Iso-Mor-Conj-over-K-prop}(2)
\end{remark}

\begin{remark}
If $E$ and $F$ are simple graphs that each have a finite number of vertices and an infinite number of edges, and we consider the statements (1)--(7) of Theorem~\ref{class-unital-inf-LPA-thm} for an arbitrary field $K$, then we have the following implications
$$(2) \Longleftarrow (1) \iff (3) \iff (4) \iff (5) \iff (6) \iff (7).$$
If, in addition, $K$ is a field with no free quotients, then we also have the implication $(2) \implies (1)$ and all seven statements are equivalent.  We will see later in Example~\ref{K-theory-counter-ex} that when $K = \Q$ (and hence $K$ has free quotients) it is possible for $(2)$ to hold while $(1)$ does not.
\end{remark}

The following corollary shows that if $E$ is a simple graph with a finite number of vertices and an infinite number of edges, then all of the algebraic $K$-groups $\{ K_n^\textnormal{alg} (L_K(E)) \}_{n \in \Z}$ are determined by the pair $(K_0^\textnormal{alg} (L_K(E)), |E^0_\textnormal{sing}|)$; and moreover, when $K$ is a field with no free quotients, all the algebraic $K$-groups are determined by  the pair $(K_0^\textnormal{alg} (L_K(E)), K_1^\textnormal{alg} (L_K(E)))$.

\begin{corollary} \label{K0-K1-determine-Kn-cor}
Let $E$ and $F$ be simple graphs that each have a finite number of vertices and an infinite number of edges, let $K$ be a field, and consider the following statements:
\begin{itemize}
\item[$(i)$] $K_0^\textnormal{alg} (L_K(E)) \cong K_0^\textnormal{alg} (L_K(F))$ and $|E^0_\textnormal{sing}| = |F^0_\textnormal{sing}|$. 
\item[$(ii)$] $K_n^\textnormal{alg} (L_K(E)) \cong K_n^\textnormal{alg} (L_K(F))$ for all $n \in \Z$.
\item[$(iii)$] $K_0^\textnormal{alg} (L_K(E)) \cong K_0^\textnormal{alg} (L_K(F))$ and $K_1^\textnormal{alg} (L_K(E)) \cong K_1^\textnormal{alg} (L_K(F))$.
\end{itemize}
In general, we have the following implications:
$$(i) \implies (ii) \implies (iii).$$
If, in addition, $K$ is a field with no free quotients, then the implication $(iii) \implies (i)$ also holds, so that all three statements are equivalent when $K$ is a field with no free quotients.
\end{corollary}

\begin{proof}
The implication $(i) \implies (ii)$ follows from the equivalence $(a) \iff (b)$ of Theorem~\ref{field-does-not-matter-thm}, and the fact that algebraic $K$-theory is a Morita equivalence invariant.  The implication $(ii) \implies (iii)$ is trivial.  If $K$ is a field with no free quotients, then the implication $(iii) \implies (i)$ follows from Corollary~\ref{K1-tell-number-singular-cor}.
\end{proof}

\begin{remark} \label{Kn-groups-unknown-rem}
If $E$ and $F$ are simple graphs that each have a finite number of vertices and an infinite number of edges, and if $(i)$, $(ii)$, and $(iii)$ are the statements of Corollary~\ref{K0-K1-determine-Kn-cor}, we will see later in Example~\ref{K-theory-not-determined-ex} that when $K = \Q$ (and hence $K$ has free quotients) it is possible for $(iii)$ to hold when $(ii)$ does not, and thus $(iii) \centernot \implies (ii)$ in general.  It is unknown to the authors at this time whether  the hypothesis that $K$ has no free quotients is necessary to obtain $(ii)$ implies $(i)$, or whether this happens for any field $K$.
\end{remark}

%%%%%%%%%%%%%%%%%%%%%%%%%%%%%%%%%%%%%%%%%%%%%%%%%%%%%%
\section{Classification of unital simple Leavitt path algebras} \label{unital-class-sec}
%%%%%%%%%%%%%%%%%%%%%%%%%%%%%%%%%%%%%%%%%%%%%%%%%%%%%%

In this section we discuss how our classifications from the previous section fit into the program to classify the unital simple Leavitt path algebras up to Morita equivalence.   We pay particular attention to the invariant used, and discuss how the number of singular vertices must be included in the invariant when the underlying field $K$ has free quotients.

It follows from the dichotomy for simple Leavitt path algebras  that if $L_K(E)$ is simple, then $L_K(E)$ is either purely infinite (if $E$ contains a cycle) or $L_K(E)$ is ultramatricial (if $E$ contains no cycles).  We mention that whether $L_K(E)$ is purely infinite or ultramatricial can also be read off from the $K_0$-group: $K_0^\textnormal{alg,+} (L_K(E)) = K_0^\textnormal{alg} (L_K(E))$ if and only if $L_K(E)$ is purely infinite.  Thus if we consider the $K_0^\text{alg}$-group to include its natural ordering (i.e., use the pair $(K_0^\textnormal{alg} (L_K(E)), K_0^\textnormal{alg,+} (L_K(E)) )$) in the invariant, then the invariant determines whether $L_K(E)$ is purely infinite or ultramatricial.

The situation when $L_K(E)$ is ultramatricial is fairly trivial, as the following lemma shows.

\begin{lemma}
If $E$ is a unital simple Leavitt path algebra that is ultramatricial, then $L_K(E) \cong M_n(K)$ for some $n \in \N$, and $L_K(E)$ is Morita equivalent to $K$.
\end{lemma}

\begin{proof}
In this case $E$ contains no cycles, and since $L_K(E)$ is also simple and unital, $E$ has a finite number of vertices and $E$ is a simple graph.  However, since $E$ has no cycles and a finite number of vertices, $E$ has a sink.  The property that all singular points are reachable by every other vertex in a simple graph then implies that $E$ has exactly one sink and no infinite emitters.  Hence $E$ is a finite graph with no cycles and exactly one sink, from which we may conclude that $L_K(E) \cong M_n(K)$ for some $n \in \N$, and $L_K(E)$ is Morita equivalent to $K$. 
\end{proof}

If $L_K(E)$ is a unital simple Leavitt path algebra that is purely infinite, we may do some additional simplifications as we consider the classification problem.  Since Move~(S) preserves Morita equivalence of the Leavitt path algebra, and since $E$ contains a finite number of vertices and a cycle, we may repeatedly apply Move~(S) to $E$ without changing the Morita equivalence class of the associated Leavitt path algebra, and after a finite number of applications we obtain a graph with no sources.  Hence there is no loss in generality if we reduce the problem to considering purely infinite unital simple Leavitt path algebras of graphs with no sources.

\begin{theorem}[cf.~Theorem~1.25 of \cite{ALPS}] \label{Abrams-et-al-thm}
Suppose $E$ and $F$ are finite graphs, $K$ is any field, and that the Leavitt path algebras  $L_K(E)$ and $L_K(F)$ are purely infinite and simple.  If the following two conditions hold:
\begin{itemize}
\item[(1)] $K_0^\textnormal{alg} (L_K(E)) \cong K_0^\textnormal{alg} (L_K(F))$, and
\item[(2)] $\sgn (\det (I-A_E^t)) = \sgn (\det (I-A_F^t))$,
\end{itemize}
then $E \sim_M F$, and $L_K(E)$ is Morita equivalent to $L_K(F)$.
\end{theorem}

\begin{remark}
In \cite{ALPS} the authors describe the moves generating the equivalence relation $\sim_M$ slightly differently.  They use the moves (S), (O), and (I) as we do here, but instead of (R) they use a move called ``contraction" (with inverse move called ``expansion"), which they describe in \cite[Definition~1.6]{ALPS}.  One can show, with just a little bit of work, that the equivalence relation generated by (S), (O), (I), and ``contraction" is the same as the  equivalence relation $\sim_M$ generated by (S), (O), (I), and (R).
\end{remark}

\begin{remark}
Observe that when $E$ and $F$ have no sources, $E \sim_M F$ implies $E$ may be transformed in $F$ using the moves (O), (I), (R), and their inverses.
\end{remark}

\begin{remark}
It is unknown whether Condition~(2) in Theorem~\ref{Abrams-et-al-thm} is necessary, and it is currently the goal of many researchers to determine whether this ``sign of the determinant condition" can be removed from the theorem.
\end{remark}

We can combine Theorem~\ref{Abrams-et-al-thm} with the results of this paper to obtain the following result summarizing the current status of the classification of unital simple Leavitt path algebras.

\begin{theorem} \label{simple-unital-ME-class-thm}
Let $K$ be any field, and let $L_K(E)$ and $L_K(F)$ be unital simple Leavitt path algebras over $K$ that are purely infinite.
\begin{itemize}
\item[(1)] If $E$ and $F$ both have a finite number of edges, and if 
$$K_0^\textnormal{alg} (L_K(E)) \cong K_0^\textnormal{alg} (L_K(F))$$ and $$\sgn (\det (I-A_E^t)) = \sgn (\det (I-A_F^t)),$$ then $L_K(E)$ is Morita equivalent to $L_K(F)$.
\item[(2)] If one of $E$ and $F$ has a finite number of edges, and the other has an infinite number of edges, then $L_K(E)$ and $L_K(F)$ are not Morita equivalent.
\item[(3)] If $E$ and $F$ both have an infinite number of edges, then $L_K(E)$ is Morita equivalent to $L_K(F)$ if and only if $$K_0^\textnormal{alg} (L_K(E)) \cong K_0^\textnormal{alg} (L_K(F)) \quad \text{ and } \quad |E^0_\textnormal{sing}| = |F^0_\textnormal{sing}|.$$  In addition, if $E$ and $F$ both have an infinite number of edges, and if $K$ is a field with no free quotients, then $L_K(E)$ is Morita equivalent to $L_K(F)$ if and only if $$ \qquad \qquad K_0^\textnormal{alg} (L_K(E)) \cong K_0^\textnormal{alg} (L_K(F)) \text{ and } K_1^\textnormal{alg} (L_K(E)) \cong K_1^\textnormal{alg} (L_K(F)).$$
\end{itemize}
\end{theorem}

\begin{proof}
Since $L_K(E)$ and $L_K(F)$ are unital simple Leavitt path algebras, $E$ and $F$ are simple graphs with a finite number of vertices.  Statement (1) is Theorem~\ref{Abrams-et-al-thm}.  To deduce statement (2), suppose, without loss of generality, hat $E$ has an infinite number of edges and $F$ has a finite number of edges.  Because $L_K(F)$ is purely infinite and simple, $F$ is simple and contains a cycle, and the cofinality of $F$ implies that $F$ has no sinks.  Since $F$ has no sinks and $F$ has a finite number of edges, $F$ has no singular vertices. However, $E$ is a graph with a finite number of vertices and an infinite number of edges, so $E$ contains at least one singular vertex.  It follows from Corollary~\ref{same-number-singular-cor} that $L_K(E)$ and $L_K(F)$ are not Morita equivalent.  Statement (3) follows from Theorem~\ref{field-does-not-matter-thm}(a),(b) and Theorem~\ref{class-unital-inf-LPA-thm}(1),(2).
\end{proof}

\begin{remark}
As described in the introduction to this section, in the classification of unital simple Leavitt path algebras, the ultramatricial case is completely understood (and fairly trivial), and it is only the purely infinite case that remains to be solved.  Theorem~\ref{simple-unital-ME-class-thm} is almost a complete classification up to Morita equivalence for unital simple Leavitt path algebras.  The only missing piece is determining whether the ``sign of the determinant condition" in (1) is necessary.  Once this question is settled, and once the proper biconditional statement is determined for (1), we will have a complete classification.
\end{remark}

\begin{remark}
If it is ultimately proven the ``sign of the determinant condition" is unnecessary in the finite graph case, then Theorem~\ref{simple-unital-ME-class-thm} shows that 
\begin{equation} \label{invariant-1-eq}
(K_0^\textnormal{alg}(L_K(E)), |E^0_\textnormal{sing}| )
\end{equation}
will be a complete Morita equivalence invariant for a simple unital Leavitt path algebra $L_K(E)$.

If instead it is discovered that the ``sign of the determinant condition" is necessary in the finite graph case (i.e., when $E$ and $F$ are finite simple graphs, $L_K(E)$ Morita equivalent to $L_K(F)$ implies that $\sgn (\det (I-A_E^t)) = \sgn (\det (I-A_E^t))$), then we will need to include the sign of the determinant as part of the invariant.  If this turns out to be the case, we could define the sign of any matrix with an entry of $\pm \infty$ to be some formal symbol $\circledast$, so that if $E$ is a graph with an infinite emitter, then $\sgn (\det (I-A_E^t)) = \circledast$.  In this case, Theorem~\ref{simple-unital-ME-class-thm} shows that 
\begin{equation} \label{invariant-2-eq}
(K_0^\textnormal{alg}(L_K(E)), |E^0_\textnormal{sing}|, \sgn (\det (I-A_E^t) ) )
\end{equation}
will be a complete Morita equivalence invariant for a simple unital Leavitt path algebra $L_K(E)$.  One should, in particular, observe that the information provided by $\sgn (\det (I-A_E^t)$ is superfluous if $E$ has an infinite emitter.  This is admittedly some notational trickery to produce a single expression for the invariant, but it does avoid having to write separate invariants for graphs with an infinite number of edges and graphs with a finite number of edges.  (We also mention that in this case it may be convenient to take $\circledast$ to be $\infty$ or $-\infty$ and think of these determinants as being infinite, or it may be convenient to take $\circledast$ to be $0$, so that the sign of the determinant is always in the set $\{-1, 0, 1\}$ and matrices with infinite entries are analogous to noninvertible finite matrices, which also have determinant zero.)

It follows from Corollary~\ref{K1-tell-number-singular-cor} that when $K$ is a field with no free quotients, one can replace $|E^0_\textnormal{sing}|$ by $K_1^\textnormal{alg}(L_K(E))$ in each of the invariants displayed in \eqref{invariant-1-eq} and \eqref{invariant-2-eq}.  However, this substitution cannot be made, and still produce a complete invariant, when the field $K$ is allowed to have free quotients.

It is also worth pointing out that the two invariants $(K_0^\textnormal{alg}(L_K(E)), |E^0_\textnormal{sing}| )$ and $(K_0^\textnormal{alg}(L_K(E)), |E^0_\textnormal{sing}|, \sgn (\det (I-A_E^t) ) )$ depend only on the graph ---indeed, Proposition~\ref{K-theory-comp-prop}(b) shows that the group $K_0^\textnormal{alg}(L_K(E))$ is independent of the field $K$.  This is to be contrasted with the two invariants $(K_0^\textnormal{alg}(L_K(E)), K_1^\textnormal{alg}(L_K(E)))$ and $(K_0^\textnormal{alg}(L_K(E)), K_1^\textnormal{alg}(L_K(E)), \sgn (\det (I-A_E^t) ) )$, which involve both the graph $E$ and the field $K$, due to the fact that the field $K$ appears in the computation of $K_1^\textnormal{alg}(L_K(E))$ (see Proposition~\ref{K-theory-comp-prop}(b)).
\end{remark}

%%%%%%%%%%%%%%%%%%%%%%%%%%%%%%%%%%%%%%%%%%%%%%%%%%%%%%
\section{The Cuntz splice} \label{CS-sec}
%%%%%%%%%%%%%%%%%%%%%%%%%%%%%%%%%%%%%%%%%%%%%%%%%%%%%%

In this section we consider a popular construction, known as the ``Cuntz splice", which changes the sign of $\det (I-A_E^t)$ when $E$ is a finite graph.  The Cuntz splice allows one to successfully remove the ``sign of the determinant condition" in the classification of simple $C^*$-algebras of finite graphs.  However, the Cuntz splice has not been able to serve the same purpose in the classification of Leavitt path algebras, because it is currently not known whether the Cuntz splice preserves the Morita equivalence class of the associated Leavitt path algebra.

\begin{definition}[Move (CS): The Cuntz Splice] 
Let $E = (E^0, E^1, r_E, s_E)$ be a graph and let $v \in E^0$ be a vertex that is the base point of at least two simple cycles. (Recall that a cycle $\alpha$ is simple if $r(\alpha_i) \neq r(\alpha)$ for all $1 \leq i \leq |\alpha|-1$.) 
Define a graph $F = (F^0, F^1, r_F, s_F)$ by $F^0 = E^0 \cup \{ v_1, v_2 \}$, $F^1 = E^1 \cup \{ e_1, e_2, f_1, f_2, h_1, h_2\}$, and let $r_F$ and $s_F$ extend $r_E$ and $s_E$, respectively, and satisfy 
$$s_F(e_1) = v, s_F(e_2) = v_1, s_F(f_1) = v_1,  s_F(f_2) = v_2, s_F(h_1) = v_1, s_F(h_2) = v_2, $$
and 
$$r_F(e_1) = v_1, r_F(e_2) = v, r_F(f_1) = v_2, r_F(f_2) = v_1,  r_F(h_1) = v_1, r_F(h_2) = v_2.$$
We say that $F$ is obtained by applying Move~(CS) to $E$ at $v$.
\end{definition}  

The following is an example of the Cuntz splice performed at vertex $\star$. 

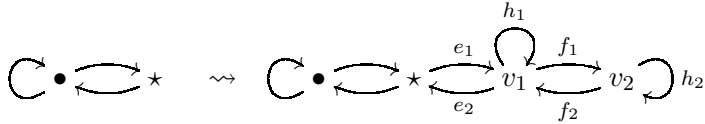
$$\xymatrix{ 
		\bullet \ar@(dl,ul) \ar@/^0.5em/[r] & \star \ar@/^0.5em/[l]
	} 
	\quad
	\rightsquigarrow
	\quad
	\quad
	\xymatrix{ 
		\bullet \ar@(dl,ul) \ar@/^0.5em/[r] & \star \ar@/^0.5em/[l] \ar@/^0.5em/[r]^{e_1} & v_1 \ar@/^0.5em/[l]^{e_2} \ar@/^0.5em/[r]^{f_1} \ar@(ul,ur)^{h_1} & v_2 \ar@/^0.5em/[l]^{f_2} \ar@(ur,dr)^{h_2}
	}
$$

\begin{remark}
If $E$ is a purely infinite simple graph with a finite number of vertices and no sources, then every vertex of $E$ is the base point of two simple cycles.  Thus in this case the Cuntz splice may be performed at any vertex of $E$.
\end{remark}

\begin{proposition} \label{effects-CS-prop}
Let $E$ be a graph, and let $F$ be a graph formed by performing the Cuntz splice to any vertex of $E$ that is the base point of at least two simple cycles.
\begin{itemize}
\item[(1)] If $E$ is a simple graph, then $F$ is a simple graph.
\item[(2)] If $K$ is any field, then $$\qquad \quad K_0^\textnormal{alg} (L_K(E)) \cong K_0^\textnormal{alg} (L_K(F)) \text{ and } K_1^\textnormal{alg} (L_K(E)) \cong K_1^\textnormal{alg} (L_K(F)).$$
\item[(3)] If $E$ is a finite graph, then $$\det (I - A_F^t) = - \det (I-A_E^t).$$
\end{itemize}
\end{proposition}

\begin{proof}
For (1) it is straightforward to see that that the Cuntz splice preserves cofinality, Condition~(L), and the existence of paths from vertices to singular vertices.  Thus if $E$ is simple, so is $F$.

For (2), let us begin by decomposing $E^0 = E^0_\textnormal{reg} \cup E^0_\text{sing}$ and writing the vertex matrix of $E$ in block form
$$A_E = \begin{pmatrix} B_E & C_E \\ * & * \end{pmatrix}.$$
If $F$ is formed by performing a Cuntz splice to $E$ at a regular vertex, then (noting that the two added vertices are regular vertices in $F$), we see that the vertex matrix of $F$ has the form
$$A_F =
\left( \begin{array}{cc|ccc|ccc} 
 								1 & 1 & 0 & 0 & \cdots & 0 & 0 & \cdots \\
 								1 & 1 & 1 & 0 & \cdots & 0 & 0 & \cdots \\
 								\hline 
 								0 & 1 & & & & & & \\
 								0 & 0 & & B_E & & & C_E &  \\ 
								\vdots & \vdots & & & & & &  \\ 
								\hline
								0 & 0 & & & & & &  \\ 
								0 & 0 & & * & & & * &  \\ 
								\vdots & \vdots & & & & & &  \\ 
 						  \end{array} \right).
						  $$
Thus the algebraic $K$-theory of $L_K(F)$ is obtained by considering the kernel and cokernel of the matrix
$$ \left( \begin{array}{cc|ccc} 
 								0 & 1 & 0 & 0 & \cdots  \\
 								1 & 0 & 1 & 0 & \cdots   \\
 								\hline 
 								0 & 1 & & &  \\
 								0 & 0 & & B_E^t-I &  \\ 
								\vdots & \vdots & & &  \\ 
								\hline
								0 & 0 & & &   \\ 
								0 & 0 & & C_E^t &   \\ 
								\vdots & \vdots & & &   \\ 
 						  \end{array} \right) { \text{equivalent} \atop \longleftrightarrow}
						  \left( \begin{array}{cc|ccc} 
 								1 & 0 & 0 & 0 & \cdots  \\
 								0 & 1 & 0 & 0 & \cdots   \\
 								\hline 
 								0 & 0 & & &  \\
 								0 & 0 & & B_E^t-I &  \\ 
								\vdots & \vdots & & &  \\ 
								\hline
								0 & 0 & & &   \\ 
								0 & 0 & & C_E^t &   \\ 
								\vdots & \vdots & & &   \\ 
 						  \end{array} \right).
						  $$
However, the $2 \times 2$ identity in the upper-left-hand corner has no effect on the kernel and cokernel, so that  
$$K_0^\textnormal{alg} (L_K(F)) \cong \coker \left( \left( \begin{smallmatrix} B_E^t-I \\ C_E^t \end{smallmatrix} \right) : \Z^{E^0_\textnormal{reg}} \to \Z^{E^0} \right) \cong K_0^\textnormal{alg} (L_K(E)),$$ and 
\begin{align*}
K_1^\textnormal{alg} (L_K(F)) &\cong \ker \left( \left( \begin{smallmatrix} B_E^t-I \\ C_E^t \end{smallmatrix} \right) : \Z^{E^0_\textnormal{reg}} \to \Z^{E^0} \right) \\
&\qquad \qquad  \oplus \coker \left( \left( \begin{smallmatrix} B_E^t-I \\ C_E^t \end{smallmatrix} \right) : (K_{1}^{\textnormal{alg} } (K) )^{E^0_\textnormal{reg}} \to ( K_{1}^{\textnormal{alg} } (K) )^{E^0} \right) \\
&\cong K_1^\textnormal{alg} (L_K(E)).
\end{align*}

If $F$ is formed by performing a Cuntz splice to $E$ at a singular vertex, then (noting that the two added vertices are regular vertices in $F$), we see that the vertex matrix of $F$ has the form
$$A_F =
\left( \begin{array}{cc|ccc|ccc} 
 								1 & 1 & 0 & 0 & \cdots & 0 & 0 & \cdots \\
 								1 & 1 & 0 & 0 & \cdots & 1 & 0 & \cdots \\
 								\hline 
 								0 & 0 & & & & & & \\
 								0 & 0 & & B_E & & & C_E &  \\ 
								\vdots & \vdots & & & & & &  \\ 
								\hline
								0 & 1 & & & & & &  \\ 
								0 & 0 & & * & & & * &  \\ 
								\vdots & \vdots & & & & & &  \\ 
 						  \end{array} \right).
						  $$
Thus the algebraic $K$-theory of $L_K(F)$ is obtained by considering the kernel and cokernel of the matrix
$$ \left( \begin{array}{cc|ccc} 
 								0 & 1 & 0 & 0 & \cdots  \\
 								1 & 0 & 0 & 0 & \cdots   \\
 								\hline 
 								0 & 0 & & &  \\
 								0 & 0 & & B_E^t-I &  \\ 
								\vdots & \vdots & & &  \\ 
								\hline
								0 & 1 & & &   \\ 
								0 & 0 & & C_E^t &   \\ 
								\vdots & \vdots & & &   \\ 
 						  \end{array} \right) { \text{equivalent} \atop \longleftrightarrow}
						  \left( \begin{array}{cc|ccc} 
 								1 & 0 & 0 & 0 & \cdots  \\
 								0 & 1 & 0 & 0 & \cdots   \\
 								\hline 
 								0 & 0 & & &  \\
 								0 & 0 & & B_E^t-I &  \\ 
								\vdots & \vdots & & &  \\ 
								\hline
								0 & 0 & & &   \\ 
								0 & 0 & & C_E^t &   \\ 
								\vdots & \vdots & & &   \\ 
 						  \end{array} \right).
						  $$
However, the $2 \times 2$ identity in the upper-left-hand corner has no effect on the kernel and cokernel, and as above we see that $K_0^\textnormal{alg} (L_K(F)) \cong K_0^\textnormal{alg} (L_K(E))$ and $K_1^\textnormal{alg} (L_K(F)) \cong K_1^\textnormal{alg} (L_K(E))$.

For (3), we see that if $E$ is a finite graph, then 
$$A_F =
\left( \begin{array}{cc|ccc} 
 								1 & 1 & 0 & 0 \ \cdots \\
 								1 & 1 & 1 & 0 \ \cdots  \\
 								\hline 
 								0 & 1 & &   \\
 								0 & 0 & & A_E \quad   \\ 
								\vdots & \vdots & &    \\ 
 						  \end{array} \right) \ \text{ and } \ \ I-A_F^t =
\left( \begin{array}{cc|ccc} 
 								0 & -1 & 0 & 0 \ \cdots \\
 								-1 & 0 & -1 & 0 \ \cdots  \\
 								\hline 
 								0 & -1 & &   \\
 								0 & 0 & & I-A_E^t   \\ 
								\vdots & \vdots & &    \\ 
 						  \end{array} \right).
						  $$
Since adding multiples of one row (respectively, column) to another row (respectively, column) does not change the determinant of a matrix, we see that $I-A_F^t$ has the same determinant as the matrix 
$$
\left( \begin{array}{cc|ccc} 
 								0 & -1 & 0 & 0 \ \cdots \\
 								-1 & 0 & 0 & 0 \ \cdots  \\
 								\hline 
 								0 & 0 & &   \\
 								0 & 0 & & I-A_E^t   \\ 
								\vdots & \vdots & &    \\ 
 						  \end{array} \right)
$$
and hence $\det (I-A_F^t) = \det \left( \begin{smallmatrix} 0 & -1 \\ -1 & 0 \end{smallmatrix} \right) \cdot \det (I-A_E^t) = - \det (I-A_E^t)$.

\end{proof}

\begin{remark}
It is known that if $E$ is a graph with a finite number of vertices, and $F$ is the graph formed by performing the Cuntz splice at a vertex of $E$ that supports two simple cycles, then $C^*(E)$ is strongly Morita equivalent to $C^*(F)$.  It is currently the effort of many researchers to determine whether $L_K(E)$ is Morita equivalent to $L_K(F)$; that is, to determine whether the Cuntz splice preserves the Morita equivalence class of the associated Leavitt path algebra.  Using our classification result, the following corollary shows that this is true if $E$ is a simple graph with a finite number of vertices and an infinite number of edges.  However, the case when $E$ is a finite simple graph is still open.
\end{remark}

\begin{proposition} \label{Cuntz-splice-preserves-Morita-eq-inf-emit-prop}
Let $E$ be a simple graph with a finite number of vertices and an infinite number of edges.  If $F$ is the graph formed by performing a Cuntz splice at any vertex of $E$, then $E \sim_M F$.  In addition, if $K$ is any field, then $L_K(E)$ is Morita equivalent to $L_K(F)$.
\end{proposition}

\begin{proof}
It follows from Proposition~\ref{effects-CS-prop} that performing the Cuntz splice preserves the simplicity of the graph and does not change the $K_0^\textnormal{alg}$-group.  It is also clear that the Cuntz splice does not change the number of singular vertices in a graph.  Therefore, if $K$ is any field, Theorem~\ref{field-does-not-matter-thm} implies that $L_K(E)$ is Morita equivalent to $L_K(F)$.
\end{proof}

\begin{remark}
If $E$ is a finite simple graph with at least one cycle and with $\det (I-A_E^t) \neq 0$, and if $F$ is the graph obtained by performing a Cuntz splice at a vertex of $E$ that supports at least two simple cycles, then it is always the case that $E$ is not move equivalent to $F$.  This is because when the moves (O), (I), and (R) are performed on a finite graph with no sources, the sign of $\det (I-A_E^t)$ is preserved, but when the Cuntz splice is performed on $E$ it changes the sign of $\det (I-A_E^t)$.  It is rather remarkable that when $E$ contains an infinite emitter (and there is no determinant to consider), the Cuntz splice can then be obtained by the moves (S), (O), (I), and (R).
\end{remark}

\begin{corollary}
Let $E$ be a simple graph with a finite number of vertices and an infinite number of edges, and let $K$ be any field.   If $F$ is the graph formed by performing a Cuntz splice at any vertex of $E$, then $K_n^\textnormal{alg}(L_K(E)) \cong K_n^\textnormal{alg}(L_K(F))$ for all $n \in \N \cup \{ 0 \}$.
\end{corollary}

\begin{remark}
Although Proposition~\ref{effects-CS-prop} shows that the Cuntz splice always preserves the $K_0^\textnormal{alg}$-group and the $K_1^\textnormal{alg}$-group, it is unknown if it preserves the higher algebraic $K$-groups.  Thus a weaker question than asking whether the Cuntz splice preserves Morita equivalence of the associated Leavitt path algebra would be to ask the following:

\smallskip

\noindent \textbf{Question:} If $E$ is a graph, $K$ is a field, and $F$ is the graph formed by performing the Cuntz splice to any vertex of $E$ that is the base point of at least two simple cycles, then is it true that $K_n^\textnormal{alg}(L_K(E)) \cong K_n^\textnormal{alg} (L_K(F))$ for all $n \in \N \cup \{ 0 \}$?

\smallskip

It would be interesting to answer this question even in the case that $E$ is a finite simple graph with no sinks or sources.

\end{remark}

\begin{remark}
If one could show that the Cuntz splice preserves Morita equivalence of Leavitt path algebras of finite simple graphs with at least one cycle, then the ``sign of the determinant condition" could be removed from Theorem~\ref{Abrams-et-al-thm} and Theorem~\ref{simple-unital-ME-class-thm} by the following argument: If $E$ and $F$ are finite simple graphs with $L_K(E)$ and $L_K(F)$ purely infinite, then $K_0^\textnormal{alg} (L_K(E) ) \cong K_0^\textnormal{alg} (L_K(F))$ implies that the $\det (I-A_E^t)$ and $\det (I-A_F^t)$ are either both zero (in the case that the $K_0^\textnormal{alg}$-groups are torsion groups) or $\det (I-A_E^t)$ and $\det (I-A_F^t)$ are both nonzero (in the case that the $K_0^\textnormal{alg}$-groups have a nonzero free part).    Thus $\det (I-A_E^t)$ and $\det (I-A_F^t)$ either have the same sign or opposite signs.  If $\sgn (\det (I-A_E^t)) = \sgn (\det (I-A_F^t))$, then Theorem~\ref{Abrams-et-al-thm} implies that $E \sim_M F$ and $L_K(E)$ is Morita equivalent to $L_K(F)$.  If $\sgn (\det (I-A_E^t)) = -\sgn (\det (I-A_F^t))$, then we may perform the Cuntz splice to any vertex of $E$ to obtain $\widetilde{E}$.  It then follows from Proposition~\ref{effects-CS-prop} that $\sgn (\det (I-A_{\widetilde{E}}^t)) = -\sgn (\det (I-A_E^t)) = \sgn (\det (I-A_F^t))$, and Theorem~\ref{Abrams-et-al-thm} implies that $\widetilde{E} \sim_M F$ and $L_K(\widetilde{E})$ is Morita equivalent to $L_K(F)$.  Hence if the Cuntz splice preserves Morita equivalence, we would also have $L_K(E)$ is Morita equivalent to $L_K(F)$.
\end{remark}

\begin{remark} \label{CS-on-L-infty-rem}
An important test case for the Cuntz splice question on finite graphs is the graph $E_2$ and the graph $E_2^-$ obtained by performing a Cuntz splice to $E_2$.

$$E_2 \qquad 
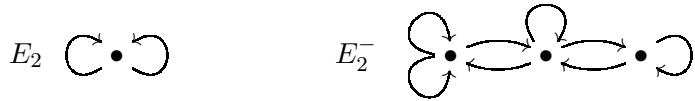
\xymatrix{ 
		\bullet \ar@(dl,ul) \ar@(dr,ur) 
	} 
	\qquad
	\qquad
	\quad
	\qquad E_2^- \qquad
	\xymatrix{ 
		\bullet \ar@(l,u) \ar@(l,d) \ar@/^0.5em/[r] & \bullet \ar@/^0.5em/[l] \ar@/^0.5em/[r] \ar@(ul,ur) & \bullet \ar@/^0.5em/[l] \ar@(ur,dr)
	}
$$

$ $

\noindent If $K$ is a field, the Leavitt path algebra of $E_2$ is the Leavitt algebra $L_K(2)$, and the Leavitt path algebra of $E_2^-$ is often denoted $L_K(2)^-$.  It is currently an open question as to whether $L_K(2)$ and $L_K(2)^-$ are Morita equivalent.

If we consider the analogous question for the Leavitt algebra $L_K(\infty)$, we have the graphs
$$E_\infty \qquad 
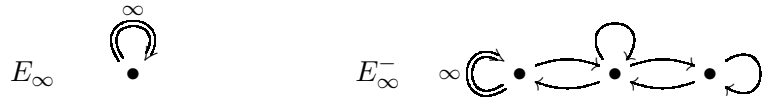
\xymatrix{ 
		\bullet \ar@{=>}@(ul,ur)[]^{\infty} 
	} 
	\qquad
	\qquad
	\quad
	\qquad E_\infty^- \quad
	\xymatrix{ 
		\bullet \ar@{=>}@(dl,ul)[]^{\infty} \ar@/^0.5em/[r] & \bullet \ar@/^0.5em/[l] \ar@/^0.5em/[r] \ar@(ul,ur) & \bullet \ar@/^0.5em/[l] \ar@(ur,dr)
	}
\smallskip
\smallskip
$$

\noindent and for any field $K$ the Leavitt path algebra of $E_\infty$ is the Leavitt algebra $L_K(\infty)$, and the Leavitt path algebra of $E_\infty^-$ is denoted $L_K(\infty)^-$.  It follows from Proposition~\ref{Cuntz-splice-preserves-Morita-eq-inf-emit-prop} that $L_K(\infty)$ is Morita equivalent to $L_K(\infty)^-$, and that $E_\infty \sim_M E_\infty^-$.  Thus it is possible to get from the graph $E_\infty$ to the graph $E_\infty^-$ using the moves (O), (I), (R), and their inverses.  (Since there are no sources, we do not need to consider Move (S).)  We will show in Example~\ref{E-infty-CS-ex} how to accomplish this.  Before we do so, however, it will be convenient for us to describe a couple of moves that are obtained by performing combinations of (O), (I), and (R).
\end{remark}

\begin{definition}[Move (C): Collapse] \label{Collapse-def}
Let $E = (E^0, E^1, r_E, s_E)$ be a graph with finitely many vertices, and let $v \in E^0$ be a regular vertex that is not the base point of a cycle of length one.  Define a graph $F = (F^0, F^1, r_F, s_F)$ by $F^0 = E^0 \setminus \{ v \}$,
$$
F^1 = \left( E^1 \setminus (r^{-1}(v) \cup s^{-1}(v)) \right) \cup  \{ [ef] : e \in r^{-1}(v) \text{ and } f \in s^{-1}(v) \},
$$
and let the range and source maps of $F$ extend those of $E$ and satisfy $r_F([ef]) = r_E(f)$ and $s_F([ef]) = s_E(e)$. 
We call $F$ the graph formed by \emph{collapsing $E$ at the vertex $v$}, and we also say $F$ is formed by performing Move~(C) to $E$ at the vertex $v$.
\end{definition}

\begin{definition}[Move (T): Transitive Closure)] \label{moveT-def}
Let $E = (E^0, E^1, r_E, s_E)$ be a graph and let $\alpha = \alpha_1 \alpha _2 \cdots \alpha_n$ be a path in $E$ and suppose that there are infinitely many edges from $s_E(\alpha_1)$ to $r_E(\alpha_1)$. 
Let $F = (F^0, F^1, r_F, s_F)$ be the graph with vertex set $F^0 := E^0$, edge set
$$
	F^1 = E^1 \cup \{ \alpha^m : m \in \N \},
$$
and range and source maps $r_F$ and $s_F$ that extend those of $E$ and have $r_{F}(\alpha^m) = r_{E}(\alpha)$ and $s_{F}(\alpha^m) = s_{E}(\alpha)$.
We call $F$ the graph formed by \emph{making $E$ transitive at the path $\alpha$}, and we also say $F$ is the graph formed by performing Move~(T) to $E$ at the path $\alpha$.
\end{definition}

\begin{remark}
It is shown in \cite[Theorem~5.2 and Theorem~5.4]{Sor} that if $F$ is obtained by performing either Move~(C) or  Move~(T) to $E$, then $E \sim_M F$.   In fact, the proofs show that each of the moves, Move~(C) and Move~(T), may be obtained by performing sequences of the three moves (O), (I), and (R), and the proofs describe exactly how this may be done. 
\end{remark}

\begin{example} \label{E-infty-CS-ex}
We will show how to transform the graph $E_\infty$ into the graph $E_\infty^-$ using the moves (O), (I), (R), and their inverses, as discussed in Remark~\ref{CS-on-L-infty-rem}.  To simplify the computation, we shall make use of the composite moves (C) and (T) discussed above.  

Beginning with $E_\infty$ we first take the edges coming out of the single vertex and partition them into two sets, one with a single edge and one with countably many edges, and perform an outsplitting with respect to this partition.  Second, we choose a single edge $e$ with source and range equal to the bottom vertex, partition the edges coming out of the bottom vertex into the the singleton set $\{ e \}$ and the set of all other edges coming out of the bottom vertex, and perform an outsplitting with respect to this partition.  (Note that since $s(e) = r(e)$, the edge $e$ gets split into $e_1$ and $e_2$.)  Third, we choose an edge $f$ going from the bottom vertex to the upper-left vertex, partition the edges coming out of the bottom vertex into the the singleton set $\{ f \}$ and the set of all other edges coming out of the bottom vertex, and perform an outsplitting with respect to this partition.  Fourth, we collapse at the center vertex.

$$
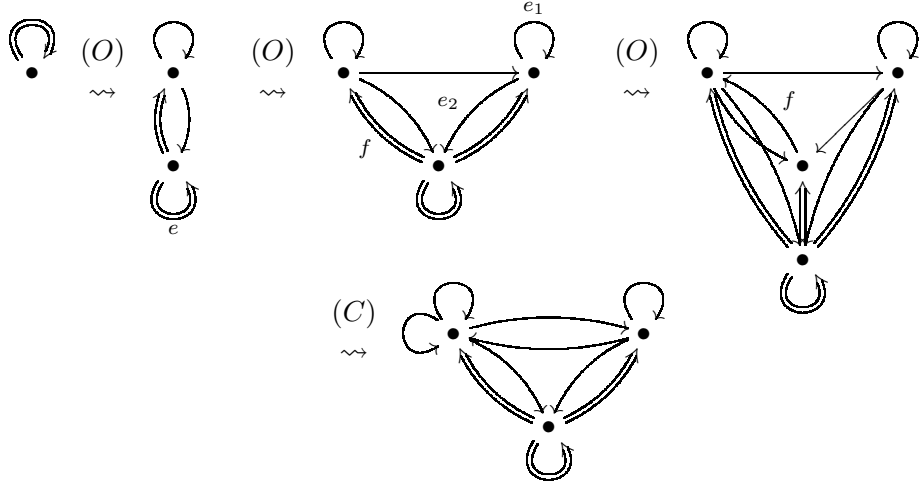
\xymatrix{ 
		\bullet \ar@{=>}@(ul,ur) 
	} 
	\quad
	{(O) \atop \rightsquigarrow}
	\quad
	\xymatrix{  \bullet \ar@(ul,ur) \ar@/^/[d]\\
		\bullet \ar@{=>}@(dl,dr)_e \ar@{=>}@/^/[u]
	} 
	\qquad
	{(O) \atop \rightsquigarrow}
	\quad
	\xymatrix{  \bullet \ar@(ul,ur) \ar@/^/[dr]  \ar[rr] & & \bullet \ar@(ul,ur)^{e_1}  \ar@/_/[dl]_<>(.6){e_2}  \\
		& \bullet \ar@{=>}@(dl,dr) \ar@{=>}@/^/[ul]^f \ar@{=>}@/_/[ur] &
	} 
	\qquad 
	{(O) \atop \rightsquigarrow}
	\quad
	\xymatrix{  \bullet \ar@(ul,ur) \ar@/^/[ddr]  \ar[rr] \ar@/_/[rd] & & \bullet \ar@(ul,ur)  \ar@/_/[ddl] \ar[dl]  \\
		& \bullet \ar@/_/[lu]_<>(0.4)f & \\
		& \bullet \ar@{=>}@(dl,dr) \ar@{=>}@/^/[uul] \ar@{=>}@/_/[uur] \ar@{=>}[u] &
	} 
$$

$$
         \qquad
	{(C) \atop \rightsquigarrow}
	\qquad
	\xymatrix{  \bullet \ar@(ul,dl) \ar@(ul,ur)  \ar@/^/[dr]  \ar@/^/[rr] & & \bullet \ar@(ul,ur)  \ar@/_/[dl] \ar@/^/[ll]   \\
		& \bullet \ar@{=>}@(dl,dr) \ar@{=>}@/^/[ul] \ar@{=>}@/_/[ur] &
	} 
$$

$ $

$ $

In the other direction, we begin with the graph $E_\infty^-$, partition the edges coming out of the upper-right vertex into $\{ e, g \}$ and $\{ f \}$, and perform an outsplitting with respect to this partition.  (Note that since $s(f) = r(f)$, the edge $f$ gets split into $f_1$ and $f_2$.)  Second, we collapse at the middle vertex.  Third, we choose a path $hi$ and perform Move~(T) to produce infinitely many edges from the bottom vertex to the upper-right vertex.  Fourth, we choose a path $jk$ and perform Move~(T) to produce infinitely many edges from the bottom vertex to the upper-left vertex.  We observe that this is the same graph we obtained above, and thus we have shown how to turn $E_\infty$ into $E_\infty^-$ using our moves.

$$\xymatrix{ 
	  \bullet \ar@(dl,ul) \ar@/^/[r] & \bullet \ar@(u,r)^f \ar@/^/[l]^e \ar@/^/[d]^g  \\
	  & \ar@/^/[u] \bullet \ar@{=>}@(dl,dr) 
	}
	\quad 
	{(O) \atop \rightsquigarrow}
	\quad
	\xymatrix{  \bullet \ar@(ul,ur) \ar[rr] \ar@/^/[rd] & & \bullet \ar@(ul,ur)^{f_1} \ar[dl]_{f_2}  \\
		& \bullet \ar@/^/[lu]^<>(0.45)e \ar@/^/[d]^<>(0.3)g & \\
		& \bullet \ar@{=>}@(dl,dr) \ar@/^/[u] \ar@/_/[uur] &
	}  
	\quad
	{(C) \atop \rightsquigarrow}
	\qquad
	\xymatrix{  \bullet \ar@(ul,dl) \ar@(ul,ur) \ar@/^/[dr] \ar@/^/[rr] & & \bullet \ar@(ul,ur)  \ar@/_/[dl] \ar@/^/[ll]   	\\		& \bullet \ar@{=>}@(dl,dr)_h \ar@/_/[ur]_i \ar@/^/[ul] &
	} 
$$

$ $

$ $

$$
         \qquad \qquad \qquad
	{(T) \atop \rightsquigarrow}
	\qquad
	\xymatrix{  \bullet \ar@(ul,dl) \ar@(ul,ur)  \ar@/^/[dr] \ar@/^/[rr] & & \bullet \ar@(ul,ur)  \ar@/_/[dl] \ar@/^/[ll]   \\
		& \bullet \ar@{=>}@(dl,dr)_j \ar@{=>}@/_/[ur]  \ar@/^/[ul]^k &
	} 
	         \qquad
	{(T) \atop \rightsquigarrow}
	\qquad
	\xymatrix{  \bullet \ar@(ul,dl) \ar@(ul,ur)  \ar@/^/[dr]  \ar@/^/[rr] & & \bullet \ar@(ul,ur)  \ar@/_/[dl] \ar@/^/[ll]   \\
		& \bullet \ar@{=>}@(dl,dr) \ar@{=>}@/^/[ul] \ar@{=>}@/_/[ur] &
	} 
$$

\end{example}

%%%%%%%%%%%%%%%%%%%%%%%%%%%%%%%%%%%%%%%%%%%%%%%%%%%%%%
\section{Classification up to isomorphism} \label{isomorphism-sec}
%%%%%%%%%%%%%%%%%%%%%%%%%%%%%%%%%%%%%%%%%%%%%%%%%%%%%%

In \cite[Theorem~2.5]{ALPS} the authors were able to use their classification up to Morita equivalence for simple Leavitt path algebras of finite graphs to obtain a classification up to isomorphism.  This was accomplished by their work in \cite[Proposition~2.4]{ALPS} where they apply a result of Huang to the shift spaces of the graphs in order to find a flow equivalence that induces a given automorphism of the $K_0$-group.  Huang's result can be found in \cite[Theorem~1.1]{Hua2} with details given in \cite[Theorem~2.15]{Hua1}.  Unfortunately, when our graphs have an infinite number of edges, the ``shift spaces" of the graphs have an infinite alphabet and therefore are not shift spaces in the traditional sense, so Huang's result does not apply.  Although we are unable to obtain a classification up to isomorphism for all unital simple Leavitt path algebras of infinite graphs, we can do so in the special case when the class of the unit in the $K_0^\textnormal{alg}$-group is an automorphism invariant.

\begin{definition}
If $G$ is a group, we say that an element $g \in G$ is \emph{automorphism invariant} if $\phi (g) = g$ for all $\phi \in \aut G$.
\end{definition}

\begin{remark}
For any group $G$, we see that the identity element $0 \in G$ is always automorphism invariant.  However, there are nonzero examples: In the group $\Z_4$ the element $[2] \in \Z_4$ is the only nonzero element of order $2$, and hence $[2]$ is automorphism invariant.  Likewise, in the group $\Z_{4} \oplus \Z$, the element $([2], 0 )$ is automorphism invariant. 
\end{remark}

\begin{remark}
Note that for any element of a group, the property of being an automorphism invariant is invariant under isomorphism; in other words, if $G$ and $H$ are groups, $\psi : G \to H$ is a group isomorphism, and $g \in G$ is automorphism invariant, then $\psi(g) \in H$ is also automorphism invariant.
\end{remark}

The argument in the proof of $(2) \implies (1)$ of the following proposition was pointed out to us by Gene Abrams.  It is a direct adaptation of Cuntz's argument in \cite[Theorem~6.5]{Ro} to the algebraic setting.

\begin{proposition} \label{unital-classification-zero-class-prop}
Let $E$ and $F$ be simple graphs that each have a finite number of vertices and an infinite number of edges, and let $K$ be any field.  If $[1_{ L_{K} (E)}]_0$ is an automorphism invariant element of $K_0^\textnormal{alg} (L_K(E))$, then the following are equivalent:
\begin{itemize}
\item[(1)] $L_{K} (E) \cong L_{K} (F)$ (as rings).
\item[(2)] There exists an isomorphism $\alpha : K_0^\textnormal{alg} (L_K(E)) \to K_0^\textnormal{alg} (L_K(F))$ with $\alpha ( [1_{L_K(E)} ]_0 ) = [1_{L_K(F)}]_0$, and $|E^0_\textnormal{sing}| = |F^0_\textnormal{sing}|$.
\item[(3)] There exists an isomorphism $\alpha : K_0^\textnormal{top} (C^*(E)) \to K_0^\textnormal{top} (C^*(F))$ with $\alpha ( [1_{C^*(E)} ]_0 ) = [1_{C^*(F)}]_0$, and $K_1^\textnormal{top} (C^*(E)) \cong K_1^\textnormal{top} (C^*(F))$.
\item[(4)] $C^*(E) \cong C^*(F)$ (as $*$-algebras).
\end{itemize}
If, in addition, $K$ is a field with no free quotients, then each of the above statements is also equivalent to the following statement:
\begin{itemize}
\item[(5)] There exists an isomorphism $\alpha : K_0^\textnormal{alg} (L_K(E)) \to K_0^\textnormal{alg} (L_K(F))$ with $\alpha ( [1_{L_K(E)} ]_0 ) = [1_{L_K(F)}]_0$, and $K_1^\textnormal{alg} (L_K(E)) \cong K_1^\textnormal{alg} (L_K(F))$.
\end{itemize}
\end{proposition}

\begin{proof}  As in the first paragraph of the proof of Theorem~\ref{class-unital-inf-LPA-thm} we see that $C^*(E)$ and $C^*(F)$ are purely infinite simple $C^*$-algebras, and that $L_K(E)$ and $L_K(F)$ are purely infinite simple rings.

The equivalence $(2) \iff (3)$ follow from the equivalences in Theorem~\ref{field-does-not-matter-thm}(b)(e) together with Proposition~\ref{K-theory-comp-prop}, which shows $K_0^\textnormal{alg} (L_K(E)) \cong K_0^\textnormal{top} (C^*(E))$ via an isomorphism taking $[1_{L_K(E)}]_0$ to $[1_{C^*(E)}]_0$.  The equivalence $(3) \iff (4)$ follow from the Kirchberg-Phillips classification theorem.  The implication $(1) \implies (2)$ follows from the fact that algebraic $K$-theory is functorial and any ring isomorphism from $L_K(E)$ to $L_K(F)$ must take the unit of $L_K(E)$ to the unit of $L_K(F)$, and from Theorem~\ref{field-does-not-matter-thm}(a)(b).

To verify $(2) \implies (1)$, suppose that (2) holds.  It follows from Theorem~\ref{field-does-not-matter-thm} that $L_{K} (E)$ is Morita equivalent to $L_{K} (F)$, and \cite[Corollary~9.11]{AT1} implies that there exists a ring isomorphism $\phi : M_{\infty} \left( L_{K} (E) \right) \rightarrow M_{\infty} \left( L_{K} (F) \right)$.   (We mention that for rings $R$ and $S$ with countable sets of enough units, which include the Leavitt path algebras, $R$ is Morita equivalent to $S$ if and only if $M_\infty (R) \cong M_\infty(S)$.  This is established in \cite[Theorem~5 and Remarks~1 and 2, p. 412]{AAM}.)  For any ring $R$ and any $x \in R$, let $x \otimes E_{11} \in M_\infty(R)$ denote the matrix in $M_\infty (R)$ with $x$ in the $(1,1)$-entry and $0$ elsewhere.  Define $e := 1_{L_K(E)} \otimes E_{11} \in M_\infty(L_K(E))$ and $f := 1_{L_K(F)} \otimes E_{11} \in M_\infty(L_K(F))$, and note that both $e$ and $f$ are idempotents.  In addition, since $L_K(E) \otimes E_{11}$ is a full corner of $M_\infty(L_K(E))$, the inclusion $x \hookrightarrow x \otimes E_{11}$ induces an isomorphism $T_E : K_0^{\textnormal{alg}}(L_K(E)) \to K_0^{\textnormal{alg}}(M_\infty(L_K(E)))$.  Likewise, since $L_K(F) \otimes E_{11}$ is a full corner of $M_\infty(L_K(F))$, the inclusion $x \hookrightarrow x \otimes E_{11}$ induces an isomorphism $T_F : K_0^{\textnormal{alg}}(L_K(F)) \to K_0^{\textnormal{alg}}(M_\infty(L_K(F)))$.   Then, if we let $\phi_0 := K_0^\textnormal{alg}(\phi)$, we see $T_E^{-1} \circ \phi_0^{-1} \circ T_F \circ \alpha : K_0^\textnormal{alg}(L_K(E)) \to K_0^\textnormal{alg}(L_K(E))$ is an automorphism, and by the hypothesis that  $[1_{ L_{K} (E)}]_0$ is automorphism invariant, we have $T_E^{-1} \circ \phi_0^{-1} \circ T_F \circ \alpha ( [1_{ L_{K} (E)}]_0 ) = [1_{ L_{K} (E)}]_0$, and $T_F (\alpha ( [1_{ L_{K} (E)}]_0 )) = \phi_0(T_E ( [1_{ L_{K} (E)}]_0))$.  Thus
\begin{align*}
[\phi(e)]_0 &= \phi_0 ( [e]_0 ) = \phi_0(T_E ( [1_{ L_{K} (E)}]_0)) = T_F (\alpha ( [1_{ L_{K} (E)}]_0 )) \\
&= T_F ([1_{ L_{K} (F)}]_0 ) = [f]_0
\end{align*}
in $K_0^\textnormal{alg}(M_\infty(L_K(F)))$.  Since $L_K(F)$ is purely infinite, $M_\infty(L_K(F))$ is also purely infinite, and by \cite[Corollary~2.2]{AbrGoodParPI}, there exist $x, y \in M_{\infty} \left( L_{K} (F) \right)$ such that $xy = \phi (e)$ and $yx = f$.  If we define $\psi : e M_{\infty} \left( L_{K} (E) \right) e \rightarrow f  M_{\infty} \left( L_{K} (F) \right) f$ by $\psi ( a ) = y \phi(a) x$, then (using the fact that $xy = \phi (e)$ and $yx = f$) one can see that $\psi$ is a ring isomorphism with inverse $z \mapsto x \phi^{-1}(z) y$.  Since $L_{K}(E) \cong e M_{\infty} \left( L_{K} (E) \right) e$ and $L_{K} (F) \cong f M_{\infty} \left( L_{K} (F) \right) f$, it follows that $L_{K}(E) \cong L_{K} (F)$.

Finally, if $K$ is a field with no free quotients, then $(5) \iff (2)$ follows from Corollary~\ref{K1-tell-number-singular-cor}.
\end{proof}

\begin{remark}
We conjecture that Proposition~\ref{unital-classification-zero-class-prop} is still true when the hypothesis that $[1_{ L_{K} (E)}]_0$ is automorphism invariant is removed.
\end{remark}

\begin{remark}
Proposition~\ref{unital-classification-zero-class-prop} shows that if $E$ is a simple graph with a finite number of vertices and an infinite number of edges, if $K$ is a field, and if $[1_{ L_{K} (E)}]_0$ is automorphism invariant of $L_K(E)$, then 
\begin{equation} \label{iso-inv-1-eq}
((K_0^\textnormal{alg}(L_K(E)), [1_{ L_{K} (E)}]_0), |E^0_\textnormal{sing}|)
\end{equation}
is a complete isomorphism invariant of $L_K(E)$ among the simple unital Leavitt path algebras over $K$.  Moreover, if $K$ is also a field with no free quotients, then this invariant may be replaced by 
\begin{equation} \label{iso-inv-2-eq}
((K_0^\textnormal{alg}(L_K(E)), [1_{ L_{K} (E)}]_0), K_1^\textnormal{alg} (L_K(F))).
\end{equation}
\end{remark}

Given the isomorphism invariants of \eqref{iso-inv-1-eq} and \eqref{iso-inv-2-eq}, and the Morita equivalence invariants of \eqref{invariant-1-eq} and \eqref{invariant-2-eq}, this raises the question as to what the ranges of these invariants are --- in particular, which finitely generated abelian groups may be realized by $K_0^{\textnormal{alg}}(L_K(E))$ and which automorphism invariant elements of $K_0^{\textnormal{alg}}(L_K(E))$ may be attained as the position of $[1_{L_K(E)}]_0$, as well as which values may be attained by $|E^0_\textnormal{sing}|$ and which abelian groups may be attained as $K_1^{\textnormal{alg}}(L_K(E))$.  The following result shows that essentially all are possible, subject only to the necessary constraints on the $K_1^\textnormal{alg} (L_K(E))$ group shared by all Leavitt path algebras, as described in Proposition~\ref{K-theory-comp-prop} (cf.~Corollary~\ref{sing-vertices-SME-inv-cor}).

\begin{proposition} \label{realization-prop}
Let $G$ be a finitely generated abelian group, let $F$ be a free abelian group with $\rank F \leq \rank G$, and let $g$ be an element of $G$.  Then there exist a graph $E$ with the following properties:
\begin{itemize}
\item[(1)] Every vertex of $E$ is the base point of at least two distinct loops (i.e., for any $v \in E^0$ there exist $e, f \in E^1$ with $e \neq f$ and $s(e) = r(e) = s(f) = f(f) = v$),
\item[(2)] $E$ is transitive (so that, in particular, $L_K(E)$ is simple and purely infinite for any field $K$), 
\item[(3)] $E^0$ is finite (so that, in particular, $L_K(E)$ is unital for any field $K$), 
\item[(4)] $|E^0_\textnormal{sing}| = \rank G - \rank F$, and
\item[(5)] for any field $K$, 
\begin{align*}
(K_0^\textnormal{alg} (L_K(E)), [1_{L_K(E)}]_0) \cong (G, g)
\end{align*}
and 
\begin{align*}
K_1^\textnormal{alg} (L_K(E)) &\cong F \oplus K^{\times} / \langle x^{ d_{1} } : x \in K^{\times} \rangle \oplus \cdots \oplus K^{\times} / \langle x^{ d_{k} } : x \in K^{\times} \rangle \\
			&\qquad \qquad \oplus (K^\times)^{m}
\end{align*} where $m := \rank G$ and $( d_{1} , \dots, d_{k} )$ are the invariant factors of $G$.
\end{itemize} 
\end{proposition}

\begin{proof}
It follows from \cite[Proposition~3.6]{EKTW} that a graph satisfying Conditions (1)--(4) exists, satisfying the additional property that $$(K_0^\textnormal{top}(C^*(E)), [1_{C^*(E)}]) \cong (G,g) \qquad \text{ and } \qquad K_1^\textnormal{top} (C^*(E)) \cong F.$$  It then follows from Corollary~\ref{K-theory-for-finite-vertices-cor} and the $K$-theory computations in Proposition~\ref{K-theory-comp-prop} that (5) holds.
\end{proof}

Note that the graph of Proposition~\ref{realization-prop} has an infinite number of edges if and only if $\rank F < \rank G$.

\begin{corollary}
If $G$ is a finitely generated abelian group, $g \in G$ and $n \in \N \cup \{ 0 \}$, then there exists a graph $E$ that is transitive, has a finite number of vertices, and satisfies the following two properties:
\begin{itemize}
\item[(1)] For every field $K$, one has $K_0(L_K(E)) \cong G$ via an isomorphism taking $[1_{L_K(E)}]_0$ to $g$
\item[(2)] $|E^0_\textnormal{sing}| = n$.
\end{itemize}
(Note that $E$ has an infinite number of edges if and only if $|E^0_\textnormal{sing}| \geq 1$.)
\end{corollary}

We conclude this section by showing that if we apply the Cuntz splice to every vertex of a graph, then the class of the unit of the corresponding Leavitt path algebra is equal to the zero element in the $K_0^\textnormal{alg}$-group.  Consequently, we have classification up to isomorphism of such graphs.

\begin{proposition} \label{Cuntz-splice-zeros-out-unit-prop}
Let $E$ be a graph with a finite number of vertices, and let $\widetilde{E}$ be the graph formed by performing a Cuntz splice to every vertex of $E$.  Then for any field $K$, we have $[1_{L_K(\widetilde{E})}]_0 = 0$ in $K_0^\textnormal{alg} (L_K(E))$.
\end{proposition}

\begin{proof}
Let $v \in E^0$.  Label the portion of $\widetilde{E}$ where a Cuntz splice has been added to $v$ as follows:
$$
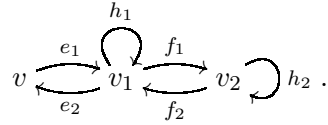
\xymatrix{ 
		 v \ar@/^0.5em/[r]^{e_1} & v_1 \ar@/^0.5em/[l]^{e_2} \ar@/^0.5em/[r]^{f_1} \ar@(ul,ur)^{h_1} & v_2 \ar@/^0.5em/[l]^{f_2} \ar@(ur,dr)^{h_2}
	}.
$$
Note that in $K_0^\textnormal{alg} (L_K(E))$ we have
$$[v_2]_0 = [f_2f_2^* + h_2h_2^*]_0 = [f_2f_2^*]_0 + [h_2h_2^*]_0 =  [f_2^*f_2]_0 + [h_2^*h_2]_0 = [v_1]_0 + [v_2]_0$$
and canceling gives $[v_1]_0 = 0$.  By similar reasoning,
\begin{align*}
[v]_0 + [v_1]_0 + [v_2]_0 &= [e_2^*e_2]_0 + [h_1^*h_1]_0 + [f_1^*f_1]_0 = [e_2e_2^*]_0 + [f_1f_1^*]_0 + [h_1h_1^*]_0 \\
&= [e_2e_2^* + f_1f_1^* + h_1h_1^*]_0 = 
[v_1]_0 = 0.
\end{align*}
It follows that in $K_0^\textnormal{alg} (L_K(E))$ we have
$$[1_{L_K(\widetilde{E})}]_0 = \left[ \sum_{w \in \widetilde{E}^0} w \right]_0 =  \sum_{w \in \widetilde{E}^0} [w ]_0 = \sum_{v \in E^0} \left( [v]_0 + [v_1]_0 + [v_2]_0 \right) = 0.$$
\end{proof}

\begin{corollary}
Let $E$ and $F$ be simple graphs with a finite number of vertices and infinite number of edges, and let $\widetilde{E}$ be the graph formed by performing a Cuntz splice to each vertex of $E$.  If $K$ is a field, then $L_K(\widetilde{E}) \cong L_K(F)$ (as rings) if and only if there exists an isomorphism $\alpha : K_0^\textnormal{alg} (L_K(\widetilde{E})) \to K_0^\textnormal{alg} (L_K(F))$ with $\alpha ( [1_{L_K(\widetilde{E})} ]_0 ) = [1_{L_K(F)}]_0$, and $|\widetilde{E}^0_\textnormal{sing}| = |F^0_\textnormal{sing}|$.
\end{corollary}

\begin{proof}
This follows from the fact that the Cuntz splice preserves simplicity, and from Proposition~\ref{unital-classification-zero-class-prop} and Proposition~\ref{Cuntz-splice-zeros-out-unit-prop}.
\end{proof}

\begin{corollary}
Let $E$ and $F$ be simple graphs with a finite number of vertices and infinite number of edges, let $\widetilde{E}$ be the graph obtained by performing a Cuntz splice to each vertex of $E$, and let $\widetilde{F}$ be the graph obtained by performing a Cuntz splice to each vertex of $F$.  If $K$ is a field, then $L_K(\widetilde{E}) \cong L_K(\widetilde{F})$ (as rings) if and only if $K_0^\textnormal{alg} (L_K(\widetilde{E})) \cong K_0^\textnormal{alg} (L_K(\widetilde{F}))$ and $|\widetilde{E}^0_\textnormal{sing}| = |\widetilde{F}^0_\textnormal{sing}|$.
\end{corollary}

\begin{proof}
This follows from the fact that the Cuntz splice preserves simplicity, from the fact every homomorphism between groups maps the identity element to the identity element, and from Proposition~\ref{unital-classification-zero-class-prop} and Proposition~\ref{Cuntz-splice-zeros-out-unit-prop}.
\end{proof}

%%%%%%%%%%%%%%%%%%%%%%%%%%%%%%%%%%%%%%%%%%%%%%%%%%%%%%
\section{Some interesting (counter)examples} \label{example-sec}
%%%%%%%%%%%%%%%%%%%%%%%%%%%%%%%%%%%%%%%%%%%%%%%%%%%%%%

In this section we consider two very interesting examples that show when $E$ is a graph with a finite number of vertices and the field $K$ has free quotients, then the Morita equivalence class of $L_K(E)$ is not determined by  $K_0^\textnormal{alg}(L_K(E))$ and  $K_1^\textnormal{alg}(L_K(E))$.  In particular, this shows the hypothesis that the field has no free quotients cannot be removed from Theorem~\ref{class-unital-inf-LPA-thm}, since statement (2) of Theorem~\ref{class-unital-inf-LPA-thm} is no longer equivalent to the other statements.  It also shows that in Corollary~\ref{K1-tell-number-singular-cor}, one has $(3) \centernot \implies (2)$ for general fields.  

The following lemma will be useful in both of our examples.

\begin{lemma} \label{K2-Q-torsion-lem}
Let $\Q$ be the field of rational numbers.  Then the following statements about the algebraic $K$-groups of $\Q$ hold:
\begin{itemize}
\item[(1)] $K_0^\textnormal{alg} (\Q) \cong \Z$.
\item[(2)] $K_1^\textnormal{alg} (\Q) \cong \Q^\times \cong \Z_2 \oplus \Z \oplus \Z \oplus \ldots$.   Consequently, $\Z \oplus K_1^\textnormal{alg} (\Q) \cong K_1^\textnormal{alg} (\Q)$.
\item[(3)] $\displaystyle K_2^\textnormal{alg}(\Q) \cong \Z_2 \oplus \bigoplus_\text{$p$ prime} \Z_p^\times$.  Consequently, $K_2^\textnormal{alg}(\Q)$ is a torsion group.
\end{itemize}
\end{lemma}

\begin{proof}
We will use some well-known facts about the algebraic $K$-theory of fields.  (We refer readers wanting more details to \cite{Gra} for a survey of the $K$-theory of fields.)  Since $\Q$ is a field, $K_0^\textnormal{alg}(\Q) \cong \Z$ and $K_1^\textnormal{alg}(\Q) \cong \Q^\times$.  It follows from Proposition~\ref{Q-times-Z-inf-prop} that $K_1^\textnormal{alg} (\Q) \cong \Z_2 \oplus \Z \oplus \Z \oplus \ldots$.  Thus $$\Z \oplus K_1^\textnormal{alg} (\Q) \cong \Z \oplus \Z_2 \oplus \Z \oplus \Z \oplus \ldots \cong \Z_2 \oplus \Z \oplus \Z \oplus \Z \oplus \ldots\cong K_1^\textnormal{alg} (\Q).$$
Hence (1) and (2) hold.  In addition, it follows from \cite[Theorem~11.6 of Tate]{Mil} that $$K_2^\textnormal{alg}(\Q) \cong \Z_2 \oplus \bigoplus_\text{$p$ prime} \Z_p^\times.$$
Thus $K_2^\textnormal{alg}(\Q)$ is a torsion group, and (3) holds.
\end{proof}

\begin{example} \label{K-theory-counter-ex}
Let $E$ and $F$ be the following graphs.

$ $

$$
E \qquad 
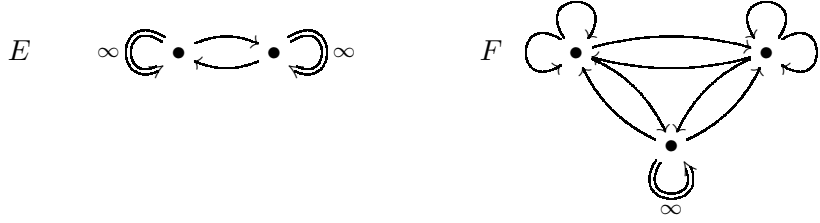
\xymatrix{ \bullet \ar@{=>}@(ul,dl)_\infty \ar@/^/[r] & \bullet  \ar@/^/[l] \ar@{=>}@(ur,dr)^\infty } \qquad \qquad F \qquad \xymatrix{  \bullet \ar@(ul,dl) \ar@(ul,ur) \ar@/^/[dr] \ar@/^/[rr] & & \bullet \ar@(ur,ul) \ar@(ur,dr)  \ar@/_/[dl] \ar@/^/[ll]   	\\		& \bullet \ar@{=>}@(dl,dr)[]_{\infty} \ar@/_/[ur] \ar@/^/[ul] &
	} 
$$

$ $

$ $

We shall consider the Leavitt path algebras $L_\Q(E)$ and $L_\Q(F)$ over the field $\Q$, and consider the $K_n^\textnormal{alg}$-groups of these algebras for $n=0,1,2$.  Note that each of $L_\Q(E)$ and $L_\Q(F)$ is purely infinite, simple, and unital.

Beginning with the graph $E$, we see that $E^0_\textnormal{reg} = \emptyset$ and $E^0 = E^0_\text{sing}$ has two elements.  Thus $K_n^\textnormal{alg} (
\Q)^{E^0_\textnormal{reg}} = 0$, and the long exact sequence of Proposition~\ref{K-theory-comp-prop}(b) implies that $K_n^\textnormal{alg} (L_\Q(E)) \cong K_n^\textnormal{alg} (\Q) \oplus K_n^\textnormal{alg} (\Q)$ for all $n \in \N \cup \{ 0 \}$.  Using the fact that $K_0^\textnormal{alg} (\Q) \cong \Z$ and $K_1^\textnormal{alg}(\Q) \cong \Q^\times$, we then have
\begin{align*}
K_0^\textnormal{alg} (L_\Q(E)) \cong \Z \oplus \Z, \quad 
K_1^\text{alg}(L_\Q(E)) \cong \Q^\times \oplus \Q^\times, 
\end{align*}  
 and $K_2^\textnormal{alg}(L_\Q(E)) \cong K_2^\textnormal{alg}(\Q) \oplus K_2^\textnormal{alg}(\Q)$.  Since  Lemma~\ref{K2-Q-torsion-lem}(3) shows that $K_2^\textnormal{alg}(\Q)$ is a torsion group, it follows that $K_2^\textnormal{alg}(L_\Q(E))$ is a torsion group.

Moving on to the graph $F$, we see the vertex matrix of $F$ is $A_F = \left( \begin{smallmatrix} 2 & 1 & 1 \\ 1 & 2 & 1\\ 1 & 1 & \infty \end{smallmatrix} \right)$, and $F$ has two regular vertices and one singular vertex.  Since the matrix $\left( \begin{smallmatrix} B^{t}_{F} - I  \\ C^t_F \end{smallmatrix} \right) =  \left( \begin{smallmatrix} 1 & 1 \\ 1 & 1 \\ 1 & 1 \end{smallmatrix} \right)$ has Smith normal form $\left( \begin{smallmatrix} 1 & 0 \\ 0 & 0 \\ 0 & 0 \end{smallmatrix} \right)$, we may use the formulae of Proposition~\ref{K-theory-comp-prop}(b) to calculate 
$$K_0^\textnormal{alg}(L_\Q(F)) \cong \coker \left( \left( \begin{smallmatrix} 1 & 0 \\ 0 & 0 \\ 0 & 0 \end{smallmatrix} \right) : \Z^2 \to \Z^3 \right) \cong \Z^2 = \Z \oplus \Z$$
and
\begin{align*}
K_1^\textnormal{alg}(L_\Q(F) )&\cong \ker \left( \left( \begin{smallmatrix} 1 & 0 \\ 0 & 0 \\ 0 & 0 \end{smallmatrix} \right) : \Z^2 \to \Z^3 \right) \oplus \coker  \left( \left( \begin{smallmatrix} 1 & 0 \\ 0 & 0 \\ 0 & 0 \end{smallmatrix} \right) : (\Q^\times)^2 \to (\Q^\times)^3 \right) \\
&\cong \Z \oplus \Q^\times \oplus \Q^\times \cong \Q^\times \oplus \Q^\times
\end{align*}
by Lemma~\ref{K2-Q-torsion-lem}(2).  Furthermore, the long exact sequence of Proposition~\ref{K-theory-comp-prop}(b) shows that
$$\xymatrix{ K_2^\textnormal{alg} (L_\Q(F)) \ar[r]^<>(.5)\phi & (\Q^\times)^2 \ar[r]^<>(0.45){ \left( \begin{smallmatrix} 1& 0 \\ 0 & 0 \\ 0 & 0 \end{smallmatrix} \right)} & (\Q^\times)^3 }
$$
is exact.  Since $\im \phi = \ker \left( \begin{smallmatrix} 1 & 0 \\ 0 & 0  \\ 0 & 0 \end{smallmatrix} \right) \cong \Q^\times  \cong \Z_{2} \oplus \Z \oplus \Z \oplus \ldots$, we see that $\im \phi$ contains non-torsion elements.  Thus $K_2^\textnormal{alg} (L_\Q(F))$ contains non-torsion elements, and $K_2^\textnormal{alg} (L_\Q(F))$ is not a torsion group.

We summarize the results of the above paragraphs here:  For the Leavitt path algebras $L_\Q(E)$ and $L_\Q(F)$, we have $K_n^\text{alg} (L_\Q(E)) \cong K_n^\text{alg} (L_\Q(F))$ for $n=0,1$, but $K_2^\text{alg} (L_\Q(E)) \ncong K_2^\text{alg} (L_\Q(F))$, because $K_2^\text{alg} (L_\Q(E))$ is a torsion group and $K_2^\text{alg} (L_\Q(F))$ is not.  In particular, the Leavitt path algebras $L_\Q(E)$ and $L_\Q(F)$ are not Morita equivalent.  
\end{example}

\begin{example} \label{K-theory-not-determined-ex}
Let $\widetilde{E}$ and $\widetilde{F}$ be the following graphs.
$$
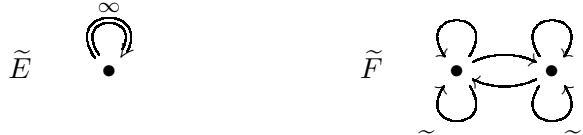
\widetilde{E} \qquad \xymatrix{ \bullet  \ar@{=>}@(ul,ur)^\infty \\} \qquad \qquad \qquad \qquad \widetilde{F} \qquad \xymatrix{ \bullet \ar@(ur,ul) \ar@(dr,dl) \ar@/^/[r] & \bullet  \ar@/^/[l] \ar@(ul,ur) \ar@(dl,dr) }
$$

$ $

\noindent We shall consider the Leavitt path algebras $L_\Q(\widetilde{E})$ and $L_\Q(\widetilde{F})$ over the field $\Q$, and compute the $K_n^\textnormal{alg}$-groups of these algebras for $n=0,1,2$.  Note that each of $L_\Q(\widetilde{E})$ and $L_\Q(\widetilde{F})$ is purely infinite, simple, and unital.

Beginning with the graph $\widetilde{E}$, we see that $\widetilde{E}^0_\textnormal{reg} = \emptyset$ and $\widetilde{E}^0 = \widetilde{E}^0_\text{sing}$ has one element.  Thus $K_n^\textnormal{alg} (\Q)^{\widetilde{E}^0_\textnormal{reg}} = 0$, and the long exact sequence of Proposition~\ref{K-theory-comp-prop}(b) implies that $K_n^\textnormal{alg} (L_\Q(\widetilde{E})) \cong K_n^\textnormal{alg} (\Q)$ for all $n \in \N$.  Using the fact that $K_0^\textnormal{alg} (\Q) \cong \Z$ and $K_1^\textnormal{alg}(\Q) \cong \Q^\times$, we then have
\begin{align*}
K_0^\textnormal{alg} (L_\Q(\widetilde{E})) \cong \Z, \ K_1^\text{alg}(L_\Q(\widetilde{E})) \cong \Q^\times
\end{align*}  
and $K_2^\textnormal{alg}(L_\Q(\widetilde{E})) \cong K_2^\textnormal{alg}(\Q)$.  We see that $K_2^\textnormal{alg}(\Q)$ is a torsion group by Lemma~\ref{K2-Q-torsion-lem}(3), and thus $K_2^\textnormal{alg}(L_\Q(\widetilde{E}))$ is a torsion group.

Moving on to the graph $\widetilde{F}$, we see the vertex matrix of $\widetilde{F}$ is $A_{\widetilde{F}} = \left( \begin{smallmatrix} 2 & 1 \\ 1 & 2 \end{smallmatrix} \right)$, and every vertex of $\widetilde{F}$ is regular so that $\widetilde{F}^0_\textnormal{reg} = \widetilde{F}^0$.  Since $A_{\widetilde{F}}^t - I = \left( \begin{smallmatrix} 1 & 1 \\ 1 & 1 \end{smallmatrix} \right)$ has Smith normal form $\left( \begin{smallmatrix} 1 & 0 \\ 0 & 0 \end{smallmatrix} \right)$ we may use the formulae of Proposition~\ref{K-theory-comp-prop}(b) to calculate 
$$K_0^\textnormal{alg}(L_\Q(\widetilde{F})) \cong \coker \left( \left( \begin{smallmatrix} 1 & 0 \\ 0 & 0 \end{smallmatrix} \right) : \Z^2 \to \Z^2 \right) \cong \Z$$
and
\begin{align*}
K_1^\textnormal{alg}(L_\Q(\widetilde{F}) ) &\cong \ker \left( \left( \begin{smallmatrix} 1 & 0 \\ 0 & 0 \end{smallmatrix} \right) : \Z^2 \to \Z^2 \right) \oplus \coker  \left( \left( \begin{smallmatrix} 1 & 0 \\ 0 & 0 \end{smallmatrix} \right) : (\Q^\times)^2 \to (\Q^\times)^2 \right) \\
&\cong \Z \oplus \Q^\times \cong \Q^{\times}
\end{align*}
by Lemma~\ref{K2-Q-torsion-lem}(2).   Furthermore, the long exact sequence of Proposition~\ref{K-theory-comp-prop}(b) shows that
$$\xymatrix{ K_2^\textnormal{alg} (L_\Q(\widetilde{F}) ) \ar[r]^<>(.5)\phi & (\Q^\times)^2 \ar[r]^{ \left( \begin{smallmatrix} 1 & 0 \\ 0 & 0 \end{smallmatrix} \right)} & (\Q^\times)^2 }
$$
is exact.  Since $\im \phi = \ker \left( \begin{smallmatrix} 1 & 0 \\ 0 & 0 \end{smallmatrix} \right) \cong \Q^\times \cong \Z_{2} \oplus \Z \oplus \Z \oplus \ldots$, we see that $\im \phi$ contains non-torsion elements.  Thus $K_2^\textnormal{alg} (L_\Q(\widetilde{F}))$ contains non-torsion elements, and $K_2^\textnormal{alg} (L_\Q(\widetilde{F}))$ is not a torsion group.

We summarize the results of the above paragraphs here:  For the Leavitt path algebras $L_\Q(\widetilde{E})$ and $L_\Q(\widetilde{F})$, we have $K_n^\text{alg} (L_\Q(\widetilde{E})) \cong K_n^\text{alg} (L_\Q(\widetilde{F}))$ for $n=0,1$, but $K_2^\text{alg} (L_\Q(\widetilde{E})) \ncong K_2^\text{alg} (L_\Q(\widetilde{F}))$, because $K_2^\text{alg} (L_\Q(\widetilde{E}))$ is a torsion group and $K_2^\text{alg} (L_\Q(\widetilde{F}))$ is not.  In particular, the Leavitt path algebras $L_\Q(\widetilde{E})$ and $L_\Q(\widetilde{F})$ are not Morita equivalent.  
\end{example}

\begin{remark}
We mention that in both Example~\ref{K-theory-counter-ex} and Example~\ref{K-theory-not-determined-ex}, another way to see that $L_\Q(E)$ and $L_\Q(F)$ are not Morita equivalent is to note that $E$ and $F$ have different numbers of singular vertices, and thus Corollary~\ref{same-number-singular-cor} implies that $L_\Q(E)$ and $L_\Q(F)$ are not Morita equivalent.
\end{remark}

There are several consequences of Example~\ref{K-theory-counter-ex} and Example~\ref{K-theory-not-determined-ex}.  We discuss these in the next few remarks and pose a number of questions prompted by these examples.

\begin{remark}
The most striking consequence of Example~\ref{K-theory-counter-ex} is that it shows the Morita equivalence class of a unital Leavitt path algebra $L_K(E)$ coming from a graph $E$ with an infinite number of edges is not determined by the pair of its algebraic $K$-groups $(K_0^\textnormal{alg} (L_K(E)), K_1^\textnormal{alg} (L_K(E)))$.  This shows that the statement in Theorem~\ref{class-unital-inf-LPA-thm}(2) cannot be added to the equivalent statements in Theorem~\ref{field-does-not-matter-thm}. Consequently, Theorem~\ref{field-does-not-matter-thm} implies that over general fields we must use the pair $(K_0^\textnormal{alg} (L_K(E)), |E^0_\textnormal{sing}|)$, rather than $(K_0^\textnormal{alg} (L_K(E)), K_1^\textnormal{alg} (L_K(E)))$, to obtain a complete Morita equivalence invariant.  

This raises the question of how to extend these classification results to nonunital simple Leavitt path algebras.  Clearly, $(K_0^\textnormal{alg} (L_K(E)), |E^0_\textnormal{sing}|)$ will no longer be the correct invariant when $E$ has infinitely many vertices --- one can readily find examples of infinite graphs $E$ and $F$, such that for any field $K$ one has $K_0^\textnormal{alg} (L_K(E)) \cong K_0^\textnormal{alg} (L_K(F))$ and $|E^0_\textnormal{sing}| = |F^0_\textnormal{sing}| = \infty$, but $L_K(E)$ is not Morita equivalent to $L_K(F)$.  On the other hand, it is still quite possible that the pair of algebraic $K$-groups $(K_0^\textnormal{alg} (L_K(E)), K_1^\textnormal{alg} (L_K(E)))$ will work to classify nonunital simple Leavitt path algebras over fields with no free quotients.  Hence, there seems to be two na\"ive ways to approach the nonunital case:  The first approach is to look for a proper generalization of the invariant $(K_0^\textnormal{alg} (L_K(E)), |E^0_\textnormal{sing}|)$ in the nonunital case.  The second approach is to restrict attention to Leavitt path algebras over fields with no free quotients (ignoring fields such as $\Q$) and attempt to use algebraic $K$-theory as the invariant, in analogy with what has worked for $C^*$-algebras.  It is unclear to the authors how to pursue the first approach, or what the proper generalization of $(K_0^\textnormal{alg} (L_K(E)), |E^0_\textnormal{sing}|)$ should be.  In addition, in either approach, it seems that new techniques must be developed, since the tools used in the unital (i.e., finite number of vertices) case involve the graph moves (S), (O), (I), and (R), and seem to rely heavily on the fact there are finitely many vertices.  

Another avenue of inquiry when looking for a complete invariant for nonunital Leavitt path algebras over arbitrary fields is to return to the finite vertex situation and ask if in this case the collection of all the algebraic $K$-groups would be sufficient to determine the Morita equivalence class.  Example~\ref{K-theory-counter-ex} and Example~\ref{K-theory-not-determined-ex} show that the $K_0^\textnormal{alg}$-group with the $K_1^\textnormal{alg}$-group do not suffice when the field $K$ has free quotients.  However, in each of these examples, one has that the $K_2^\textnormal{alg}$-groups of each of the Leavitt path algebras are different.  Is it possible that if one includes the $K_2^\textnormal{alg}$-group one obtains a complete Morita equivalence invariant for all fields?  If not, will a finite number of the algebraic $K$-groups suffice?  If not, will all the algebraic $K$-groups suffice?  We summarize these inquiries in the following question.

\smallskip

\noindent \textbf{Question~1:} Does there exist $N \in \N \cup \{ \infty \}$ with the property that whenever $E$ and $F$ are simple graphs with a finite number of vertices and an infinite number of edges, and $K$ is any field, then $K_n^\textnormal{alg} (L_K(E)) \cong K_n^\textnormal{alg} (L_K(F))$ for all $0 \leq n  < N$ implies that $L_K(E)$ is Morita equivalent to $L_K(F)$?
If one considers Corollary~\ref{K0-K1-determine-Kn-cor}, one can think of Question 1 as asking whether Corollary~\ref{K0-K1-determine-Kn-cor}(ii) provides the proper invariant for general fields (Cf.~Remark~\ref{Kn-groups-unknown-rem}).

\smallskip

Before attempting to answer such a question, however, one may wish to ask how worthwhile an answer would be.  In particular, is the collection $\{ K_n^\textnormal{alg} (L_K(E)) : n \in \N \}$ a useful invariant?  Although we can calculate $K_n^\textnormal{alg} (L_K(E))$ for $n=0,1$ as described in Proposition~\ref{K-theory-comp-prop}(b), for $n \geq 2$ we have no explicit formula and the best tool at our disposal is the long exact sequence shown in Proposition~\ref{K-theory-comp-prop}(b).  Although this long exact sequence can sometimes be useful for extracting partial information about the higher algebraic $K$-groups, in general  there are many examples of Leavitt path algebras where there is little to nothing we can say about their higher algebraic $K$-groups.  This leads to our next question.

\smallskip

\noindent \textbf{Question~2:} If $E$ is a simple graph with a finite number of vertices and an infinite number of edges, and $K$ is any field, can one give a tractable method for calculating $K_n^\textnormal{alg} (L_K(E))$ for $n \geq 2$?

\smallskip

\noindent An answer to Question~2 would even be interesting if one adds the hypothesis that $K$ is a field with no free quotients.

\end{remark}

\begin{remark} \label{number-sing-vert-rem}
In each of Example~\ref{K-theory-counter-ex} and Example~\ref{K-theory-not-determined-ex} we have pairs of  graphs with isomorphic $K_0^\textnormal{alg}$-groups and isomorphic $K_1^\textnormal{alg}$-groups, but a different number of singular vertices.  This shows that $(3) \centernot \implies (2)$ in Corollary~\ref{K1-tell-number-singular-cor} when we remove the hypothesis that $K$ is a field with no free quotients.  In particular, when $K$ has free quotients, the groups $K_0^\textnormal{alg}(L_K(E))$ and $K_1^\textnormal{alg}(L_K(E))$ do not determine the number of singular vertices in the graph $E$.  As Example~\ref{K-theory-not-determined-ex} shows, when $K$ is a field with a free quotient, then the groups $K_0^\textnormal{alg}(L_K(E))$ and $K_1^\textnormal{alg}(L_K(E))$ are not even sufficient to determine whether $E$ has no singular vertices.  Hence, if $E$ has a finite number of vertices, the groups $K_0^\textnormal{alg}(L_K(E))$ and $K_1^\textnormal{alg}(L_K(E))$ cannot, in general, tell us whether or not $E$ is finite.  This leads us to the following question

\smallskip

\noindent \textbf{Question~3:} If $E$ and $F$ are graphs with a finite number of vertices, $K$ is any field, and $K_n^\textnormal{alg}(L_K(E)) \cong K_n^\textnormal{alg}(L_K(F))$ for all $n \in \N \cup \{ 0 \}$, then is it necessarily the case that $|E^0_\textnormal{sing}| = |F^0_\textnormal{sing}|$?

\smallskip

\noindent If such an $N$ is found, then Theorem~\ref{field-does-not-matter-thm} shows us that $\{ K_n^\textnormal{alg}(L_K(E)) \}_{n=1}^N$ would be a complete Morita equivalence invariant for unital simple Leavitt path algebras $L_K(E)$ when $E$ has an infinite number of edges.  This has two advantages over using the pair $(K_n^\textnormal{alg}(L_K(E)), |E^0_\textnormal{sing}|)$; first, it is an invariant expressed entirely in terms of the algebra $L_K(E)$ without reference to properties of the particular graph $E$; and second, it is an invariant that has a chance of classifying nonunital simple Leavitt path algebras.

\end{remark}

\end{document}